\documentclass[3p, 11pt]{elsarticle}

\usepackage{color}
\usepackage[colorlinks,linkcolor=blue,citecolor=blue]{hyperref}
\usepackage{amsmath,amsfonts,amssymb,amsthm,stmaryrd,bm}
\usepackage{mathrsfs} 
\usepackage{graphicx,epstopdf,float,subfigure}
\usepackage[margin=10pt,font=small,labelfont=bf,labelsep=period]{caption}

\usepackage{multirow,booktabs}

\numberwithin{equation}{section} 
\newtheorem{theorem}{Theorem}[section]
\newtheorem{lemma}{Lemma}[section]

\newtheorem{remark}{Remark}[section]
\newtheorem{example}{Example}[section]
\biboptions{sort&compress}

\def \d {\mathrm{d}}

\newcommand{\bb}{\boldsymbol}

\usepackage{titlesec}
\titleformat{\subsection}{\normalfont\bfseries} {\thesubsection}{1em}{}[]

\begin{document}
	
	\begin{frontmatter}
		
		\title{A posteriori error estimation for an interior penalty virtual element method of Kirchhoff plates}

		\author{Fang Feng}
		\ead{ffeng@njust.edu.cn}
		\address{School of Mathematics and Statistics, Nanjing University of Science and Technology, Nanjing, China}

		\author{Yuming Hu}
        \ead{202421511168@smail.xtu.edu.cn}
        \author{Yue Yu\footnote{Corresponding author.}}		
		\ead{terenceyuyue@xtu.edu.cn}
		\address{Hunan Key Laboratory for Computation and Simulation in Science and Engineering, Key Laboratory of Intelligent Computing and Information Processing of Ministry of Education, National Center for Applied Mathematics in Hunan, School of Mathematics and Computational Science, Xiangtan University, Xiangtan, Hunan 411105, China}	

        \author{Jikun Zhao}
        \ead{jkzhao@zzu.edu.cn}
         \address{School of Mathematics and Statistics, Zhengzhou University, Zhengzhou 450001, PR China}

\begin{abstract}
  In this paper, we develop a residual-type a posteriori error estimation for an interior penalty virtual element method (IPVEM) for the Kirchhoff plate bending problem. Building on the work in \cite{FY2023IPVEM}, we adopt a modified discrete variational formulation that incorporates the $ H^1 $-elliptic projector in the jump and average terms. This allows us to simplify the numerical implementation by including the $ H^1 $-elliptic projector in the computable error estimators.
  We derive the reliability and efficiency of the a posteriori error bound by constructing an enriching operator and establishing some related error estimates that align with $C^0$-continuous interior penalty finite element methods.
  As observed in the a priori analysis, the interior penalty virtual elements exhibit similar behaviors to $C^0$-continuous elements despite its discontinuity. This observation extends to the a posteriori estimate since we do not need to account for the jumps of the function itself in the discrete scheme and the error estimators.
  As an outcome of the error estimator, an adaptive VEM is introduced by means of the mesh refinement strategy with the one-hanging-node rule. Numerical results from several benchmark tests confirm the robustness of the proposed error estimators and show the efficiency of the resulting adaptive VEM.
\end{abstract}
		

\end{frontmatter}
	
	%
	%

	\section{Introduction}

Virtual element methods (VEMs) generalize standard finite element methods to accommodate general polytopal meshes. Initially proposed and analyzed in \cite{Beirao-Brezzi-Cangiani-2013}, with subsequent foundational work in \cite{Ahmad-Alsaedi-Brezzi-2013,Beirao-Brezzi-Marini-2014}, VEMs offer advantages in handling complex geometries and high-regularity solutions compared to traditional finite element methods.
For the plate bending problem or the biharmonic equation, Brezzi and Marini introduced an $ H^2 $-conforming VEM in \cite{Brezzi-Marini-2013}, establishing convergence results in the $ H^2 $ norm. Chinosi and Marini further improved the method in \cite{Chinosi-Marini-2016} by deriving optimal order error estimates in the $ L^2 $ and $ H^1 $ norms through VEM enhancement techniques. Nonconforming VEMs have also been developed, including the $ C^0 $-continuous $ H^2 $-nonconforming VEM presented in \cite{Zhao-Chen-2016} and the Morley-type VEM discussed in \cite{Zhao-Zhang-Chen-2018}. Chen and Huang extended the approach to fully $ H^m $-nonconforming VEMs for $ 2m $-th order problems in any dimension in \cite{Chen-HuangX-2020}.
Recently, the interior penalty  virtual element method (IPVEM) has been explored in \cite{ZMZW2023IPVEM} for the biharmonic equation.  Compared with the conforming or nonconforming VEMs for the biharmonic problem, the IPVEM has advantages
in reducing the number of degrees of freedom (DoFs), which can be regarded as a combination of the virtual element space and discontinuous Galerkin scheme.
 This method was adapted for the fourth-order singular perturbation problem in \cite{ZZ24}, with some adaptations to the original space. These modifications encompass alterations in the definition of the $H^1$-type projection and the selection of the DoFs. In contrast, Ref.~\cite{FY2023IPVEM} employs the original IPVEM formulation from \cite{ZMZW2023IPVEM} to address the fourth-order singular perturbation problem, incorporating techniques from the modified Morley finite element method as described in \cite{WXH06}.

Another point to be mentioned is that due to the large flexibility of the meshes in VEMs, there is no exceptional treatment for hanging nodes, which has great superiority in mesh refinement. Hence, a posteriori error analysis of VEMs is well worth studying. For the fourth-order problem, a residual-based a posteriori error estimate  for the $C^1$-continuous VEM was developed in \cite{ChenHuangLin2022avem}, addressing both the reliability and efficiency of the error bound.  Recently, the lowest-order nonconforming or Morley-type virtual element has been applied to the biharmonic equation, with a residual-based a posteriori error estimator derived in \cite{Carsten2023Morley}. It's worth noting that the a posteriori error analysis in \cite{Carsten2023Morley} circumvents any trace of second derivatives by using some computable conforming companion operator, which significantly reduces the computational complexity of the error estimators. This reduction is facilitated by the special property of the interpolation operator for the Morley virtual element (see Lemma \ref{lem:H2ncIh}).
Recently, the medius error estimates for the interior penalty virtual element method for the biharmonic equation are proposed in \cite{ZJZZ24}.

 In this paper, we consider a posteriori error estimation for the interior penalty virtual element method applied to Kirchhoff plate bending problems. Unlike the jump and average terms in \cite{ZMZW2023IPVEM}, we include the $ H^1 $-elliptic projector $ \Pi_h^{\nabla} $ in the penalty terms, as proposed in \cite{FY2023IPVEM}. This modification enables the design of computable error estimators that incorporate the $ H^1 $-elliptic projector, thereby simplifying the implementation process.
Following the approach in \cite{Huang2021YuMedius}, we construct an enriching operator to establish the reliability of the a posteriori error estimators. We then develop a residual-type a posteriori error estimation for the modified IPVEM applied to Kirchhoff plates.
It is worth noting that although the interior penalty virtual element method is discontinuous, as observed in the a priori error estimate, the interior penalty virtual elements exhibit behavior similar to $ C^0 $-continuous elements. This ``quasi-$ C^0 $ continuity'' can be extended to the a posteriori estimate since we do not need to account for the jumps of the function itself in the discrete scheme or the error estimators.

The remainder of the paper is structured as follows. We begin by introducing the continuous variational problem and presenting some useful results in VEM analysis in Section \ref{sec:cvariationalProb}.
Section \ref{sec:IPVEM} describes the interior penalty virtual element method. We also incorporate inhomogeneous boundary conditions not covered in previous references in the section for numerical experiments. In order to simplify the implementation, we add the elliptic projector $\Pi_h^{\nabla}$ for all $v,w$ in the jump and average terms. In Section \ref{enriching}, we construct an enriching operator and present the corresponding error estimates. Section \ref{posteriori_analysis} develops the a posteriori analysis for the IPVEM, providing both upper and lower bounds. We do not include the jumps of the function itself in the computable posteriori error estimators despite the discontinuity of the interior penalty virtual element.
 Meanwhile, we make a comparison with the analysis of Morley-type virtual elements in \cite{Carsten2023Morley}.
 An adaptive IPVEM is presented in Section \ref{Sec:numerical} with several benchmark tests performed to illustrate the robustness of the residual-type a posteriori error estimator and demonstrate the efficiency of the adaptive IPVEM.

\section{The continuous variational problem} \label{sec:cvariationalProb}

Let $\Omega$ be a bounded  polygonal domain of $\mathbb{R}^2$ with boundary $\partial \Omega$.  For $f\in L^2(\Omega)$, we consider the following boundary value problem of the Kirchhoff plate:
\begin{equation}\label{strongform}
	\left\{\begin{array}{ll}
		\Delta^2u = f & {\rm in}~~\Omega,\\
		u = 0, \quad \dfrac{\partial u}{\partial \bm n}=0& {\rm on}~~\partial \Omega,
	\end{array}\right.
\end{equation}
where $\bm{n} = (n_1,n_2)^T$ is the unit outer normal to $\partial \Omega$, $\Delta$ is the standard Laplacian operator. Throughout the article we assume $u \in H^{k+1}(\Omega)$ with $k\ge 2$.

We first introduce some notations and symbols frequently used in this paper. For a bounded Lipschitz domain $D$ of dimension $d ~(d=1,2)$, the symbol $( v, w)_D = \int_D v w \d x$ denotes the $L^2$-inner product on $D$, $\|\cdot\|_{0,D}$ or $\|\cdot\|_{D}$ denotes the $L^2$-norm, and $|\cdot|_{s,D}$ is the $H^s(D)$-seminorm. If $D=\Omega$, for simplicity, we abbreviate $\|\cdot\|_{\Omega}$ as $\|\cdot\|$. For vectorial functions $\bm{v} = (v_1,v_2)^T$ and $\bm{w} = (w_1,w_2)^T$, the inner product is defined as $( \bm{v}, \bm{w})_D = \int_D (v_1w_1 + v_2 w_2) \d x$.
For all integer $k\ge 0$, $\mathbb{P}_k(D)$ is the set of polynomials of degree $\le k$ on $D$. The set of scaled monomials $\mathbb{M}_r(D)$ is given by
\[\mathbb  M_{r} (D):= \Big \{ \Big ( \frac{\boldsymbol x -  \boldsymbol x_D}{h_D}\Big )^{\boldsymbol  s}, \quad |\boldsymbol  s|\le r\Big \},\]
with the generic monomial  denoted by $m_{D}$, where $h_D$ is the diameter of $D$, $\boldsymbol  x_D$ is the centroid of $D$, and $r$ is a non-negative integer. For the multi-index ${\boldsymbol{s}} \in {\mathbb{N}^d}$, we follow the usual notation
\[\boldsymbol{x}^{\boldsymbol{s}} = x_1^{s_1} \cdots x_d^{s_d},\quad |\boldsymbol{s}| = s_1 +  \cdots  + s_d.\]
Conventionally, $\mathbb  M_r (D) =\{0\}$ for $r\le -1$.

Let $\{\mathcal{T}_h\}$ be a family of decompositions of $\Omega$ into polygonal elements $\{K\}$;
$h_K={\rm diam}(K)$ and $h=\max_{K\in\mathcal{T}_h}h_K$.
Let $\mathcal{V}_h$ be the union of all vertices and $\mathcal{E}_h$ be the set of all the edges in $\mathcal{T}_h$. For each element $K\in\mathcal{T}_h$, we denote by $\mathcal{V}(K)$ and $\mathcal{E}(K)$ the sets of vertices and edges of $K$, respectively.
Moreover, denote all interior vertices (resp. edges) by $\mathcal{V}_h^0$ (resp. $\mathcal{E}_h^0$). The union of boundary edges in $\mathcal{T}_h$ is denoted by $\mathcal{E}_h^\partial$.

Let $e \in \mathcal{E}_h^0$ be an interior edge shared by two neighbouring elements $K^-$ and $K^+$. The unit normal vector $\bb{n}_e$ associated with $e$ is defined as the outward unit normal of $K = K^-$ pointing from $K^-$ to $K^+$.
Let $v$ be a scalar function defined on $e$. We introduce the jump and average of $v$ on $e$ by $[v] = v^- - v^+$ and $\{v\} = \frac12 (v^- + v^+)$, where $v^-$ and $v^+$ are the traces of $v$ on $e$ from the interior and exterior of $K$, respectively. On a boundary edge, we define $[ v ] = v$ and $\{ v \} = v $. From the definition of the averages and jumps, by a direct manipulation, one has $[ v w ]=[v]\{w\}+\{v\}[w]$ for $e\in \mathcal{E}_h^0$ and $[vw] = vw = \{v\}[w] = [v]\{w\}$ for a boundary edge, leading to
\begin{equation}\label{magic}
	\sum_{e\in \mathcal{E}_h}\int_e [vw] \d s = \sum_{e\in \mathcal{E}_h^0} \int_e [v]\{w\}\d s  + \sum_{e\in \mathcal{E}_h} \int_e \{v\}[w] \d s.
\end{equation}	
Moreover, for any two quantities $a$ and $b$, ``$a\lesssim b$" indicates ``$a\le C b$" with the constant $C$ independent of the mesh size $h_K$, and ``$a\eqsim b$" abbreviates ``$a\lesssim b\lesssim a$".

The variational formulation of \eqref{strongform} with homogeneous boundary value conditions reads: Find $u\in V:=H_0^2(\Omega)$ such that
\begin{equation}\label{origion}
	a(u,v)=(f,v),\quad  v\in V,
\end{equation}
where $a(u,v) = (\nabla^2u,\nabla^2v)$.
To keep the flexibility of meshes, we will work on meshes satisfying the following assumption (cf. \cite{Brezzi-Buffa-Lipnikov-2009,Chen-HuangJ-2018}):
\begin{enumerate}
	\item [\textbf{H}.]
	For each $K\in\mathcal{T}_h$, there exists a ``virtual triangulation" $\mathcal{T}_K$
	of $K$ such that $\mathcal{T}_K$ is uniformly shape regular and quasi-uniform.
	The corresponding mesh size of $\mathcal{T}_K$ is proportional to $h_K$.
   Each edge of $K$ is a side of a certain triangle in $\mathcal{T}_K$.
\end{enumerate}
As shown in \cite{Chen-HuangJ-2018}, this condition covers the usual geometric assumptions frequently used in the context of VEMs.
Under this geometric assumption, we can establish some fundamental results in VEM analysis as used in \cite{Huang2021YuMedius}, with the constants hidden in the symbol $\lesssim$ depending only on the shape regularity and quasi-uniformity of the virtual triangulation ${\mathcal{T}_K}$.

According to the standard Dupont-Scott theory (cf. \cite{BS2008}), for all $v\in H^l(K)$ ($0\le l \le k$) there exists
a certain $q \in \mathbb{P}_{l-1}(K)$ such that
\begin{equation}\label{BHe1}
	|v - q|_{m,K} \lesssim h_K^{l - m} | v |_{l,K},\quad  m\le l.
\end{equation}

The following inequalities are very useful for our forthcoming analysis.

\begin{lemma}
	For any $K\in \mathcal{T}_h$, it holds
\begin{itemize}
  \item Trace inequality \cite{BS2008}:
  \begin{equation}
		\| v \|_{0,\partial K} \lesssim  h_K^{1/2}| v |_{1,K} +h_K^{ - 1/2}\| v \|_{0,K},\quad v \in H^1(K), \label{trace}
	\end{equation}
  \item Sobolev inequality \cite{Brenner-Guan-Sung-2017}:
  	\begin{equation}
		\| v \|_{\infty,K} \lesssim  h_K| v |_{2,K} +|v|_{1,K}+h_K^{-1}\| v \|_{0,K},\quad v \in H^2(K),\label{Sobolev}
	\end{equation}
  \item Poincar\'{e}-Friedriches inequality \cite{Brenner2003}:
  	\begin{align}
	& \| v \|_{0,K} \lesssim  h_K| v |_{1,K} + \Big|\int_{\partial K}v\,{\rm d}s\Big|,\quad v \in H^1(K).\label{poincare}
   \end{align}
\end{itemize}	
\end{lemma}

    \section{The interior penalty virtual element method} \label{sec:IPVEM}
	
	This section reviews the interior penalty virtual element space developed in \cite{ZMZW2023IPVEM}.
	
	\subsection{The IP virtual element spaces}
	
	In the construction, the authors in \cite{ZMZW2023IPVEM} first introduced a $C^1$-conforming virtual element space
	\begin{align*}
		& \widetilde{V}_{k+2}(K)=\Big\{v \in H^2(K): \Delta^2 v \in \mathbb{P}_k(K), v|_e \in \mathbb{P}_{k+2}(e),
		 \partial_{\bm n} v|_e \in \mathbb{P}_{k+1}(e), e \subset \partial K\Big\}, \quad k \ge 2.
	\end{align*}
	This local space can be equipped with the following degrees of freedom (DoFs) (cf. \cite{Brezzi-Marini-2013,Chinosi-Marini-2016}):
	\begin{itemize}
		\item  $\tilde{\bm\chi}^p:$ the values of $v$ at the vertices of $K$,
		\begin{equation}\label{dof1}
			\tilde{\chi}_z^p(v) = v(z), \quad \mbox{$z$ is a vertex of $K$}.
		\end{equation}
		
		\item  $\tilde{ \bm \chi}^g$ : the values of $h_z \nabla v$ at the vertices  of $K$,
		\begin{equation}\label{dof2}
			\tilde{\chi}_z^g(v) = h_z \nabla v(z), \quad \mbox{$z$ is a vertex of $K$},
		\end{equation}
		where $h_z$ is a characteristic length attached to each vertex $z$, for instance, the average of the diameters of the elements having $z$ as a vertex.
		\item  $\tilde{\bm \chi}^e$ : the moments of $v$ on edges up to degree $k-2$,
		\begin{equation}\label{dof3}
			\tilde{\chi}_e(v) = |e|^{-1}(m_e, v)_e, \quad m_e \in \mathbb{M}_{k-2}(e), \quad e \subset\partial K.
		\end{equation}
		\item  $\tilde{\bm \chi}^n$ : the moments of $\partial_{\bm n_e} v$ on edges up to degree $k-1$,
		\begin{equation}\label{dof4}
			\tilde{\chi}_e^n(v) = (m_e,\partial_{\bm n_e} v)_e, \quad  m_e \in \mathbb{M}_{k-1}(e), \quad e \subset \partial K.
		\end{equation}
		\item  $\tilde{\bm \chi}^K$ : the moments on element $K$ up to degree $k$,
		\begin{equation}\label{dof5}
			\tilde{\chi}_K(v) = |K|^{-1}(m_K,v)_K, \quad m_K \in \mathbb{M}_k(K).
		\end{equation}
	\end{itemize}
Referring to \cite[Lemma 3.4]{ZMZW2023IPVEM}, we have the following lemma.
\begin{lemma}[Inverse inequality]
    For any $K\in \mathcal{T}_h$, for $m=1,2$ and $0\leq s\leq m$, there holds
    \begin{equation}\label{inverse_inequality}
        |v_h|_{m,K}\lesssim h_K^{s-m}|v_h|_{s,K},\quad v_h\in \widetilde{V}_{k+2}(K).
    \end{equation}
\end{lemma}

	Given $v_h \in \widetilde{V}_{k+2}(K)$, the usual definition of the $H^1$-elliptic projection $\Pi_K^{\nabla} v_h \in$ $\mathbb{P}_k(K)$ is described by the following equations:
	\begin{equation}\label{H1def}
		\left\{\begin{aligned}
			(\nabla \Pi_K^{\nabla} v_h, \nabla q)_K  & = (\nabla v_h, \nabla q)_K, \quad q \in \mathbb{P}_k(K), \\
			\sum\limits_{z \in \mathcal{V}_K} \Pi_K^{\nabla} v_h(z) & = \sum\limits_{z \in \mathcal{V}_K} v_h(z),
		\end{aligned}\right.
	\end{equation}
	where $\mathcal{V}_K$ is the set of the vertices of $K$.
	By checking the right-hand side of the integration by parts formula
	\[(\nabla v_h, \nabla q)_K = -(v_h, \Delta q)_K + \sum\limits_{e\subset \partial K} \int_e v_h \partial_{\bm n_e} q\, \mathrm{d} s,\]
one can find that this elliptic projection can be computed using the DoFs in \eqref{dof1}--\eqref{dof5}.
	However, the goal of the IPVEM is to make $\Pi_K^\nabla v_h$ computable by only using the DoFs of $H^1$-conforming virtual element spaces given by (cf. \cite{Beirao-Brezzi-Cangiani-2013,Ahmad-Alsaedi-Brezzi-2013})
	\[V_h^{1,c}(K): = \{ v \in H^1(K): \Delta v|_K \in \mathbb{P}_{k - 2}(K)~~{\text{in}}~~K,\quad  v|_{\partial K} \in \mathbb{B}_k(\partial K) \},\]
	where
	\[\mathbb{B}_k(\partial K): = \{ v \in C(\partial K):  v|_e \in \mathbb{P}_k(e), \quad e \subset \partial K \},\]
	and the corresponding global virtual element space is
	\[
	V_h^{1,c}=\{v\in C(\overline{\Omega}):  v|_K\in V_h^{1,c}(K), \, K\in \mathcal{T}_h\}\cap H_0^1(\Omega).
	\]
	To do so, Ref.~\cite{ZMZW2023IPVEM} considered an approximation of the right-hand side by some numerical formula. The modified $H^1$-projection is defined as
	\[
	\left\{\begin{aligned}
		(\nabla \Pi_K^{\nabla} v_h, \nabla q)_K & = -(v_h, \Delta q)_K + \sum\limits_{e\subset \partial K} Q_{2k-1}^e( v_h \partial_{\bm n_e} q), \quad q \in \mathbb{P}_k(K), \notag\\
		\sum_{z \in \mathcal{V}_K} \Pi_K^{\nabla} v_h(z) & =\sum_{z \in \mathcal{V}_K} v_h(z),\label{MH1}
	\end{aligned}\right.
	\]
	with
	\[ Q_{2k-1}^e v := |e| \sum\limits_{i=0}^k \omega_i v(\bm{x}_i^e) \approx \int_e v(s) \d s,\]
	where $(\omega_i, \bm{x}_i^e)$ are the $(k+1)$ Gauss-Lobatto quadrature weights and points with $\bm{x}_0^e$ and $\bm{x}_k^e$ being the endpoints of $e$.
	
	With the help of $\Pi_K^\nabla$, we can employ the standard enhancement technique by substituting the redundant DoFs of $v$ with those of $\Pi_K^{\nabla}v$ \cite{Ahmad-Alsaedi-Brezzi-2013,ZMZW2023IPVEM}. The local interior penalty space is defined as
	\begin{align*}
		&V_k(K)= \{v \in \widetilde{V}_{k+2}(K):  \tilde{\chi}^g(v)=\tilde{\chi}^g (\Pi_K^{\nabla} v ), \tilde{\chi}^n(v)=\tilde{\chi}^n (\Pi_K^{\nabla} v ),\\
       &\hspace{4cm}  (v,q)_K = (\Pi_K^{\nabla}v,q)_K\quad q\in \mathbb{M}_k(K) \backslash \mathbb{M}_{k-2}(K)\} \quad k\ge 2,
	\end{align*}
	The associated DoFs are given by
	\begin{itemize}
		\item  $ \bm \chi^p:$ the values of $v(z)$, $z \in \mathcal{V}_K$;
		\item $ \bm \chi^e:$ the values of $v(\boldsymbol{x}_i^e)$, $i=1,2 \cdots, k-1$, $e \subset \partial K$;
		\item $\bm \chi^K$ : the moments $|K|^{-1}( m_K, v)_K$, $m_K \in \mathbb{M}_{k-2}(K)$.
	\end{itemize}
	Furthermore, we use $V_h$ to denote the global space of nonconforming virtual element by
\begin{align*}
			V_h & = \{v \in L^2(\Omega): v|_K \in V_k(K),  ~K \in \mathcal{T}_h,  \, \text{$v$ is continuous at each Gauss-Lobatto point} \\
			& \hspace{0.2cm} \text{of interior edges and vanishes at each Gauss-Lobatto point of boundary edges}\}.
	\end{align*}
Denote by $I_h: H^2(\Omega) \to V_h$  the interpolation operator. As shown in  Lemma 3.11 of \cite{ZMZW2023IPVEM}, we can derive the following interpolation error estimate.
\begin{lemma}\label{lem:interpIPVEM}
	For any $v\in H^{\ell}(K)$ with $2\le \ell\le k+1$, it holds
	\begin{equation} \label{interp}
		|v-I_hv|_{m,K}\lesssim h_K^{\ell-m}|v|_{\ell,K},\quad K\in \mathcal{T}_h,\quad  m=0,1,2.
	\end{equation}	
\end{lemma}

	As usual, we can define the $H^2$-projection operator $\Pi_K^{\Delta}: V_k(K) \to \mathbb{P}_k(K)$ by finding the solution $\Pi_K^{\Delta} v \in \mathbb{P}_k(K)$ of
	\begin{equation}\label{def}
		\begin{cases}
			a^K (\Pi_K^{\Delta} v, q )=a^K(v, q),\quad  q \in \mathbb{P}_k(K), \\
			\widehat{\Pi_K^{\Delta} v}=\widehat{v}, \quad \widehat{\nabla \Pi_K^{\Delta} v}=\widehat{\nabla v}
		\end{cases}
	\end{equation}
	for any given $v \in V_k(K)$, where the quasi-average $\widehat{v}$ is defined by
	\[\widehat{v}=\frac{1}{|\partial K|} \int_{\partial K} v \d s .\]

\begin{lemma}\label{lem:bound}
For every $v\in V_k(K)$, there hold
\begin{align*}
 |\Pi_K^{\nabla}v|_{m,K}\lesssim |v|_{m,K} \quad \mbox{and}\quad |\Pi_K^{\Delta}v|_{m,K} \lesssim |v|_{m,K},
\end{align*}
where $m=0,1,2$.
\end{lemma}
\begin{proof}
The first inequality is presented in Corollary 3.7 of \cite{ZMZW2023IPVEM}. For the second one, it is simple to find that $|\Pi_K^{\Delta}v|_{2,K} \lesssim |v|_{2,K}$. For $m=1$, by the triangle inequality, the Poincar\'{e}-Friedrichs inequality \eqref{poincare}, the boundedness of $|\Pi_K^{\Delta}v|_{2,K}$ and the inverse inequality \eqref{inverse_inequality}, we derive
\begin{align*}
|\Pi_K^{\Delta}v|_{1,K}
&  \le |v - \Pi_K^{\Delta}v |_{1,K} + |v|_{1,K} \lesssim h_K|v - \Pi_K^{\Delta}v |_{2,K} + |v|_{1,K}\notag\\
&\lesssim h_K|v|_{2,K}+|v|_{1,K}\lesssim |v|_{1,K}.
\end{align*}
The case of $m=0$ can be deduced in the similar manner.
\end{proof}

\subsection{The discrete IP bilinear form}
	Following the standard procedure in the VEMs (namely, we need to consider the computability, the $k$-consistency \eqref{consistent} and the stability \eqref{stable}), we define the discrete bilinear form as
	\[a_h^K(v, w)=a^K(\Pi_K^{\Delta} v, \Pi_K^{\Delta} w)+ h_K^{-2} S^K(v-\Pi_K^{\Delta} v, w-\Pi_K^{\Delta} w), \quad v, w \in V_k(K),\]
	with
	\[
	S^K(v, w) = \sum\limits_{i=1}^{n_K} \chi_i(v) \chi_i(w),
	\]
	where $\{\chi_i\}$ are the local DoFs on $K$ with $n_K$ being the number of the DoFs. As was done in \cite{FY2023IPVEM}, we define
	\begin{align}
	& J_1(v,w)=   \sum_{e\in\mathcal{E}_h}\frac{\lambda_e}{|e|}\int_e\Big[\frac{\partial  \Pi_h^\nabla v}{\partial \bm{n}_e}\Big]\Big[\frac{\partial \Pi_h^\nabla w}{\partial \bm{n}_e}\Big] \d s \label{J1},\\
    & J_2(v,w)=-\sum_{e \in \mathcal{E}_h}\int_e\Big\{\frac{\partial^2\Pi_h^\nabla v}{\partial \bm{n}_e^2}\Big\}\Big[\frac{\partial \Pi_h^\nabla w}{\partial \bm{n}_e}\Big] \d s = -\sum_{e \in \mathcal{E}_h}\int_e\Big\{\frac{\partial^2\Pi_h^\nabla v}{\partial \bm{n}_e^2}\Big\}\Big[\frac{\partial  w}{\partial \bm{n}_e}\Big] \d s, \nonumber\\
    & J_3(v,w)=-\sum_{e\in\mathcal{E}_h}\int_e\Big\{\frac{\partial^2 \Pi_h^\nabla w}{\partial \bm{n}_e^2}\Big\}\Big[\frac{\partial  \Pi_h^\nabla v}{\partial \bm{n}_e}\Big] \d s = -\sum_{e\in\mathcal{E}_h}\int_e\Big\{\frac{\partial^2 \Pi_h^\nabla w}{\partial \bm{n}_e^2}\Big\}\Big[\frac{\partial v}{\partial \bm{n}_e}\Big] \d s. \nonumber
	\end{align}
	Here and below, we define the piecewise $H^1$ and $H^2$ projectors $\Pi_h^\nabla$ and $\Pi_h^\Delta$ by setting $\Pi_h^\nabla |_K=\Pi_K^\nabla$ and $\Pi_h^\Delta |_K=\Pi_K^\Delta$ for all $K\in \mathcal{T}_h$.

	 The bilinear form is
	\begin{equation}\label{BilinearAh}
		\mathcal{A}_h(u_h,v_h) = a_h(u_h,v_h)+J_1(u_h,v_h)+J_2(u_h,v_h)+J_3(u_h,v_h),
	\end{equation}
	where $a_h(v,w)=\sum_{K\in \mathcal{T}_h} a_h^K(v,w)$. The discrete scheme of IPVEM solving problem \eqref{strongform} is to find $u_h \in V_h$ such that
	\begin{equation}\label{IPVEM}
		\mathcal{A}_h(u_h,v_h) = F_h(v_h),\quad v_h\in V_h,
	\end{equation}The right-hand side is
	\[F_h(v_h): = \langle f_h,v_h\rangle
,\]
	where
	\begin{align}\label{eq:rhs}
		\langle f_h,v_h \rangle:	=
		\sum\limits_{K\in \mathcal{T}_h}\int_{K} f\,\Pi_{0,K}^k v_h \d x= \sum\limits_{K\in \mathcal{T}_h}\int_{K} \Pi_{0,K}^kf\, v_h \d x ,
	\end{align}
	with $\Pi_{0,K}^k$ being the $L^2$ projector onto $\mathbb{P}_k(K)$. 	For $f\in H^{k-2}(\Omega)$ with $k\geq 2$, we have
	\begin{align*}
		\langle f-f_h,v_h\rangle&=\sum_{K\in \mathcal{T}_h}(f,v_h-\Pi_{0,K}^kv_h)=\sum_{K\in \mathcal{T}_h}(f-\Pi_{0,K}^kf,v_h-\Pi_{0,K}^kv_h) \notag\\
		&\lesssim h^{k-2}\|f\|_{k-2}h^2|v_h|_{2,h}\lesssim h^{k}\|f\|_{k-2}|v_h|_{2,h},
	\end{align*}
	which means
	\begin{equation}\label{right}
		\|f-f_h\|_{V'_h}\lesssim h^{k} \|f\|_{k-2}.
	\end{equation}

\begin{remark}\label{rem:C0issue}
As observed in \cite{FY2023IPVEM}, despite the discontinuity, the interior penalty virtual elements display behaviors similar to $C^0$-continuous elements in the a priori estimate. We will demonstrate that this observation also holds for the a posteriori estimate. Therefore, we do not include the penalty term
\[\sum_{e\in\mathcal{E}_h}\frac{\lambda_e^3}{|e|^3}\int_e [v] [w] \d s\]
in the discrete scheme and introduce the error estimator
\begin{equation}\label{eta0}
\eta_{0,K} = \Big(\sum\limits_{e\in \mathcal{E}(K)\cap \mathcal{E}_h}  \eta_{0,e}^2\Big)^{1/2}, \qquad
		\eta_{0,e} =\frac{1}{|e|^{3/2}}\| [ u_h ]\|_{0,e}
\end{equation}
 in the a posteriori estimate, as will be discussed later.
\end{remark}

	In what follows, we define
	\begin{equation} \label{normm}
		\|w\|_h^2:=|w|_{2,h}^2+J_1(w,w).
	\end{equation}
 Referring to \cite{FY2023IPVEM}, the mesh-dependent parameter $\lambda_e\eqsim 1$. Still, we can deduce the stability result by using the same arguments.  We omit the details with the results described as follows.
	\begin{itemize}
		\item[-] $k$-consistency: for all $v \in V_k(K)$ and $q \in \mathbb{P}_k(K)$, it holds that
		\begin{equation}\label{consistent}
			a_h^K(v, q)=a^K(v, q) .
		\end{equation}
		\item[-] Stability: there exist two positive constants $\alpha_*$ and $\alpha^*$, independent of $h$, such that
		\begin{align}
			& \alpha_*a^K(v, v) \le a_h^K(v, v) \le  \alpha^* a^K(v, v) \label{stable}
		\end{align}
		for all $v \in V_k(K)$.
	\end{itemize}

	For later use, we provide the following norm equivalence.
\begin{lemma}
For every $v\in V_k(K)$, there hold
\begin{align}
& |v-\Pi_K^{\Delta}v|_{2,K} \eqsim  h_K^{-1}  \|\bm\chi(v-\Pi_K^{\Delta}v)\|_{\ell^2},   \label{Norm_Equiv_Qk_H2K} \\
& |v-\Pi_K^{\nabla}v|_{1,K} \eqsim   \|\bm\chi(v-\Pi_K^{\nabla}v)\|_{\ell^2}, \label{NormEquivH1}\\
& h_K^{-1} \|v-\Pi_K^{\nabla}v\|_{0,K} \eqsim   \|\bm\chi(v-\Pi_K^{\nabla}v)\|_{\ell^2}.\label{NormEquivL2}
\end{align}
\end{lemma}
\begin{proof}

The first and second equations are presented in \cite[Lemma 4.1]{ZZ24}. Following the approach in \cite{Huang2021YuMedius}, it is straightforward to derive the lower bound estimate $\|\bm\chi(v-\Pi_K^{\nabla}v)\|_{\ell^2} \lesssim h_K^{-1} \|v-\Pi_K^{\nabla}v\|_{0,K}$ through direct calculations.

For the upper bound estimate, according to the boundedness of $\Pi_h^\nabla$ and $\Pi_h^\Delta$ in Lemma \ref{lem:bound} and the Poincar\'{e}-Friedrichs inequality \eqref{poincare}, one has
\begin{align}
\|v - \Pi_h^{\nabla}v\|_{0,K}
& \le \|v - \Pi_h^{\Delta}v\|_{0,K} + \|\Pi_h^{\Delta}v-\Pi_h^{\nabla}v\|_{0,K} \nonumber\\
& = \|v - \Pi_h^{\Delta}v\|_{0,K} + \|\Pi_h^{\nabla}(\Pi_h^{\Delta}v - v)\|_{0,K}
 \lesssim \|v-\Pi_h^{\Delta}v\|_{0,K} \nonumber\\
& \lesssim h_K |v - \Pi_h^{\Delta}v|_{1,K} \le h_K |v - \Pi_h^{\nabla}v|_{1,K} + h_K |\Pi_h^{\Delta}(\Pi_h^{\nabla}v - v)|_{1,K} \nonumber\\
& \lesssim h_K|v - \Pi_h^{\nabla}v|_{1,K}. \label{L2H1}
\end{align}
This along with the second formula yields the desired result.
\end{proof}

	\section{Enriching operator}\label{enriching}
	
	This section focuses on proposing and analyzing an enriching operator $E_h$ from $V_h$ to $H^2$-conforming virtual element spaces.
	The enriching operator is a special kind of quasi-interpolation operator, which connects a finite-dimensional space to another one with higher-order regularity. We refer to \cite{Brenner2003,Brenner-Wang-Zhao-2004,Gudi2010,Mao-Shi-2010,Hu-Rui-Shi-2014,Brenner2005SungC0IP} and the references therein for the construction and application of such operators in finite element methods. Such an operator in the virtual element context was first proposed in \cite{Huang2021YuMedius} and applied in the medius error analysis of some nonconforming VEMs. This idea was also adopted in \cite{ChenHuangLin2022avem} for the a posteriori error estimation for a $C^1$ virtual element method of Kirchhoff
	plates.
	
	Let us recall the $H^2$-conforming virtual element space $V_k^{2,c}(K)$ introduced in \cite{Brezzi-Marini-2013,Chinosi-Marini-2016}. For $k\ge 3$, define
	\[{V_k^{2,c}}(K)= \{ v \in H^2(K): \Delta ^2v|_K \in \mathbb{P}_{k - 4}(K),~ v|_e \in \mathbb{P}_k(e),~ \partial_{\boldsymbol n} v|_e \in {\mathbb{P}_{k - 1}}(e),~e \subset \partial K \},\]
	while for the lowest order $k=2$, the space is modified as
	\[V_2^{2,c}(K)= \Big\{ v \in H^2(K): \Delta ^2v |_K = 0,~ v |_e \in \mathbb{P}_3(e),~ \partial_{\boldsymbol n}v |_e \in \mathbb{P}_1(e),~e \subset \partial K \Big\}.\]
	The DoFs are:
	\begin{itemize}
		\item The values of $v(z)$ at the vertices of $K$.
		\item The values of $h_z\partial _1v(z)$ and $h_z\partial _2v(z)$ at the vertices of $K$, where $h_z$ is a characteristic length attached to each vertex $z$, for instance, the average of the diameters of the elements having $z$ as a vertex.
		\item The moments of $v$ on edges up to degree $k-4$,
		\[\chi _e(v) = | e |^{-1}(m_e,v)_e,\quad m_e \in \mathbb{M}_{k - 4}(e).\]
		\item The moments of $\partial_{\boldsymbol n}v$ on edges up to degree $k-3$,
		\[\chi _{n_e}(v) = (m_e, \partial_{\boldsymbol n}v)_e,\quad m_e \in \mathbb{M}_{k - 3}(e).\]
		\item The moments on element $K$ up to degree $k-4$,
		\[\chi _K(v) = | K |^{-1}(m_K, v)_K,\quad m_K \in \mathbb{M}_{k - 4}(K).\]
	\end{itemize}
	Note that the above moments vanish in the lowest-order case $k=2$.

	The global virtual element spaces are then defined as
	\[V_h^{2,c}: = \Big\{ v \in H_0^2(\Omega ) \cap C^1(\overline{\Omega} ): v|_K \in V_k^{2,c}(K),~~K \in \mathcal{T}_h \Big\}.\]
	The global dual space is defined as
	\[\mathcal{N} = \text{span}\Big\{ \mathcal{N}_a,\mathcal{N}_\nabla,\mathcal{N}_e^{k - 4},\mathcal{N}_{n_e}^{k - 3},\mathcal{N}_K^{k - 4} \Big\}\]
	with global DoFs given by
	\begin{itemize}
		\item $\mathcal{N}_a$: the values at the interior vertices of $\mathcal{T}_h$,
		\[\mathcal{N}_a^z(v) = v(z),\quad z \in \mathcal{V}_h^0.\]
		\item $\mathcal{N}_\nabla$: the gradient values at interior vertices of $\mathcal{T}_h$,
		\[\mathcal{N}_\nabla ^z(v) = h_z\nabla v(z),\quad z \in \mathcal{V}_h^0.\]
		\item ${\mathcal{N}  _e^{k - 4}}$: the moments of $v$ on interior edges up to degree $k-4$,
		\[\mathcal{N} _e(v) = | e |^{ - 1}(m_e,v)_e,\quad m_e \in \mathbb{M}_{k - 4}(e),\quad e \in \mathcal{E}_h^0.\]
		\item $\mathcal{N}_{n_e}^{k - 3}$: the moments of $\partial_{\boldsymbol n}v$ on interior edges up to degree $k-3$,
		\[\mathcal{N} _{n_e}(v) = (m_e,\partial_{\boldsymbol n}v)_e,\quad m_e \in \mathbb{M}_{k - 3}(e),\quad e \in \mathcal{E}_h^0.\]
		\item $\mathcal{N}  _K^{k-4}$: the moments on element $K$ up to degree $k-4$,
		\[\mathcal{N} _K(v) = | K |^{ - 1}(m_K,v)_K,\quad m_K \in \mathbb{M}_{k-4}(K).\]
	\end{itemize}
	Relabel the d.o.f.s by a single index $i = 1, 2,  \ldots, N^{2,c} := \dim V_h^{2,c}$, which are associated with a set of shape basis functions $\{\Phi_j\}$ of $V_h^{2,c}$ such that $\mathcal{N}_i(\Phi_j) = \delta_{ij}$ for $i,j = 1, \cdots, N^{2,c}$. Then every function $v\in V_h^{2,c}$ can be expressed as
	\begin{equation}\label{gexpa}
		v(x) = \sum_{i=1}^{N^{2,c}}\mathcal{N}_i(v)\Phi_i(x).
	\end{equation}

	For a vertex $z$ of $\mathcal{T}_h$, let $\omega (z)$ denote the union of all elements in $\mathcal{T}_h$ sharing the point $z$. For an edge $e$ in $\mathcal{T}_h$, let $\omega (e)$ be the union of all elements in $\mathcal{T}_h$ sharing the edge $e$. Let $N(z)$ and $N(e)$ denote the number of elements in $\omega(z)$ and $\omega (e)$, respectively.

	For every nonconforming VEM function $\phi  \in V_h$, we shall construct an associated conforming counterpart $E_h\phi\in V_h^{2,c}\cap H_0^2(\Omega)$  by using the global expansion \eqref{gexpa}, i.e.,
	\[(E_h\phi )(x) = \sum\limits_{i = 1}^{N^{2,c}} \mathcal{N}_i (E_h\phi )\Phi _i(x).\]
	Here, the values of the DoFs are determined as follows:
	\begin{enumerate}
		\item For the values at interior vertices,
		\[\mathcal{N}_a^z(E_h\phi)=(E_h\phi )(z): = \frac{1}{N(z)}\sum\limits_{K' \in \omega (z)} \Pi_h^{\nabla} \phi |_{K'}(z) \quad z\in \mathcal{V}_h^0.\]
		\item For the gradient values at interior vertices,
		\[\mathcal{N}_\nabla ^z(E_h\phi ):=\frac{1}{N(z)}\sum\limits_{K' \in \omega (z)} h_z\nabla \Pi_h^{\nabla}\phi|_{K'}(z)\quad z\in \mathcal{V}_h^0.\]
		\item For the moments of $v$ on interior edges,
		\[\mathcal{N}_e(E_h\phi ): = \frac{1}{N(e)}\sum\limits_{K' \in \omega (e)} | e |^{ - 1}\int_e m_e\Pi_h^{\nabla}\phi |_{K'} \mathrm{d}s,\quad m_e\in \mathbb{M}_{k-4}(e) .\]
		\item For the moments of ${\partial _{\boldsymbol{n}}}v$ on interior edges,
		\[\mathcal{N}_{n_e}(E_h\phi ): = \frac{1}{N(e)}\sum\limits_{K' \in \omega (e)} \int_e m_e\partial _{\boldsymbol{n}}\Pi_h^{\nabla}\phi |_{K'} {\text{d}}s,\quad m_e \in \mathbb{M}_{k - 3}(e).\]
		\item For the moments on element $K$,
		\[\mathcal{N}_K(E_h\phi ): = |K|^{ - 1}\int_K m_K \Pi_h^{\nabla}\phi|_K {\mathrm{d}}x,\quad m_K \in \mathbb{M}_{k - 4}(K).\]
	\end{enumerate}
	
\begin{lemma}\label{lem:enriching}
For any $v\in V_h$ and $E_h v$ defined above, there holds
\begin{equation}\label{C0}
|v-E_hv|_{2,h}^2 \lesssim \sum_{K\in \mathcal{T}_h } h_K^{-2}S^K(v-\Pi_h^\nabla v, v-\Pi_h^\nabla v)  + \sum_{e\in \mathcal{E}_h}
			\Big( \frac{1}{|e|^3}\|[v]\|_{0,e}^2 + \frac{1}{|e|}\Big\|\Big[ \frac{\partial v}{\partial \bm n_e}\Big]\Big\|_{0,e}^2 \Big).
\end{equation}
In addition,
\begin{equation}\label{jump0}
\sum_{e\in \mathcal{E}_h} \frac{1}{|e|^3}\|[v]\|_{0,e}^2 \lesssim \sum_{K\in \mathcal{T}_h } h_K^{-2}S^K(v-\Pi_h^\nabla v, v-\Pi_h^\nabla v),
\end{equation}
\begin{equation}\label{jump1}
\sum_{e\in \mathcal{E}_h} \frac{1}{|e|}\Big\|\Big[ \frac{\partial v}{\partial \bm n_e}\Big]\Big\|_{0,e}^2
\lesssim \sum_{K\in \mathcal{T}_h } h_K^{-2}S^K(v-\Pi_h^\nabla v, v-\Pi_h^\nabla v) +
\sum_{e\in \mathcal{E}_h} \frac{1}{|e|}\Big\|\Big[ \frac{\partial \Pi_h^\nabla v}{\partial \bm n_e}\Big]\Big\|_{0,e}^2 .
\end{equation}
	\end{lemma}
	
	\begin{proof}
Step 1: We first prove \eqref{C0}. Consider the triangle inequality
		\begin{align}\label{triangle}
			|v-E_hv|_{2,h}
			\le |v- \Pi_h^\nabla v|_{2,h} + |\Pi_h^\nabla v - E_hv|_{2,h}.
		\end{align}
		Using the inverse inequality \eqref{inverse_inequality} and the norm equivalence \eqref{NormEquivH1}, we obtain
       \begin{align*}
           |v- \Pi_h^\nabla v|_{2,h}
       & = \sum_{K\in \mathcal{T}_h } |v- \Pi_h^\nabla v|_{2,K} \lesssim \sum_{K\in \mathcal{T}_h } h_K^{-1} |v- \Pi_h^\nabla v|_{1,K}\\
       & \lesssim \sum_{K\in \mathcal{T}_h } h_K^{-1}( S^K(v-\Pi_h^\nabla v, v-\Pi_h^\nabla v))^{1/2}.
       \end{align*}
Thus, it reduces to bound the second term on the right-hand side of \eqref{triangle}.
		Since $\Pi_h^\nabla v - E_hv\in V_k^{2,c}(K)$, by the inverse inequality and the norm equivalence in \cite{Huang2021YuMedius}, one gets
		\[|\Pi_h^\nabla v - E_hv|_{2,K} \lesssim h_K^{-2} \|\Pi_h^\nabla v - E_hv\|_{0,K} \lesssim h_K^{-1} \| \bm\chi^c(\Pi_h^\nabla v - E_hv) \|_{l^2}, \quad K \in \mathcal{T}_h.\]
		where $\bm\chi^c$ is the set of DoFs on $K$ for the $C^1$-continuous space.
		In the following, we only provide the details of bounding the values and its gradient values of the function $\Pi_h^\Delta v - E_hv$ at vertices since the argument implies the treatment for moments.
		
	Let $z\in \mathcal{V}_h^0$ be an interior vertex. Assume that $K_1 = K,K_2, \cdots ,K_L$ share the node $z$, and denote $(\Pi_h^{\nabla}\phi )_i =\Pi_h^{\nabla}\phi|_{K_i}$, where $K_i$  and  $K_{i+1}$ are two neighboring elements. We have
		\begin{align*}
			& (\Pi_h^{\nabla}v-E_hv)|_K(z) \\
			& = (\Pi_h^{\nabla}v )_1(z) - \frac{1}{L}( (\Pi_h^{\nabla}v )_1 +  \cdots  + (\Pi_h^{\nabla}v )_L)(z)  \\
			& = \frac{1}{L}(((\Pi_h^{\nabla}v)_1-(\Pi_h^{\nabla}v)_2) + ((\Pi_h^{\nabla}v )_1 - (\Pi_h^{\nabla}v )_3) +  \cdots  + ((\Pi_h^{\nabla}v )_1 - (\Pi_h^{\nabla}v )_L ))(z).
		\end{align*}
		Since $L$ is uniformly bounded and
		\[(\Pi_h^{\nabla}v )_1(z) - (\Pi_h^{\nabla}v )_j(z) = \sum\limits_{i = 1}^{j - 1} ((\Pi_h^{\nabla}v )_i - (\Pi_h^{\nabla}v )_{i + 1})(z),\]
		it suffices to consider the term $(\Pi_h^{\nabla}v )_1(z) - (\Pi_h^{\nabla}v )_2(z)$.
		Apply the inverse inequality of polynomials to get
		\begin{equation}\label{vzdiff}
				|(\Pi_h^{\nabla}v)_1(z) - (\Pi_h^{\nabla}v )_2(z)| \le \| (\Pi_h^{\nabla}v )_1 - (\Pi_h^{\nabla}v )_2 \|_{\infty ,e} \lesssim h_e^{ - 1/2}\| (\Pi_h^{\nabla}v )_1 - (\Pi_h^{\nabla}v )_2 \|_{0,e}.
		\end{equation}
		Noting that $v$ is not continuous across $e$, we denote by $v_i$ the restriction of $v$ on $K_i$. Then one has
		\begin{align*}
			\| (\Pi_h^{\nabla}v )_1 - (\Pi_h^{\nabla}v )_2 \|_{0,e}
			& = \| (\Pi_h^{\nabla}v )_1 - v_1 + v_1-v_2 + v_2 - (\Pi_h^{\nabla}v )_2 \|_{0,e} \\
			& \le \| (\Pi_h^{\nabla}v )_1 -v_1\|_{0,e}+\|v_2- (\Pi_h^{\nabla}v )_2 \|_{0,e}+\|[v]\|_{0,e}.
		\end{align*}
By the trace inequality \eqref{trace},
		\[
		\| v  - \Pi_h^{\nabla}v  \|_{0,e} \lesssim h_K^{1/2}| v  - \Pi_h^{\nabla}v |_{1,K} + h_K^{ - 1/2}\| v  - \Pi_h^{\nabla}v  \|_{0,K}.
		\]

		If $z$ is a vertex on the domain boundary, then $E_h(v)(z) = 0 = v(z)$, which gives
		\begin{align*}
			|(\Pi_h^{\nabla}v-E_hv)|_K(z) | = |\Pi_h^{\nabla}v(z) - v(z)| \le \|\Pi_h^{\nabla}v - v\|_{\infty, K}.
		\end{align*}
 Using the Sobolev inequality \eqref{Sobolev} and the inverse inequality \eqref{inverse_inequality}, we get
\begin{align*}
\|v - \Pi_h^{\nabla}v\|_{\infty, K}
& \le h_K|v - \Pi_h^{\nabla}v|_{2,K} + |v - \Pi_h^{\nabla}v|_{1,K} + h_K^{-1}\|v - \Pi_h^{\nabla}v\|_{0,K} \\
& \lesssim |v - \Pi_h^{\nabla}v|_{1,K} + h_K^{-1}\|v - \Pi_h^{\nabla}v\|_{0,K} .
\end{align*}

Therefore, for any vertex $z$ of $\mathcal{T}_h$, by using norm equivalence \eqref{NormEquivH1} and \eqref{NormEquivL2}, we obtain
\[\Big(h_K^{-1} (\Pi_h^{\nabla}v-E_hv)|_K(z)\Big)^2 \lesssim h_K^{-2}S^K(v-\Pi_h^\nabla v, v-\Pi_h^\nabla v)
		+ \sum_{e_i} \frac{1}{|e_i|^3}\|[v]\|_{0,e_i}^2 ,\]
where $e_i$ is the edge sharing by $K_i$ and $K_{i+1}$.

For the gradient value at an interior vertex $z\in \mathcal{V}^0$, one has
		\begin{align*}
			&h_z\nabla (\Pi_h^{\nabla}v-E_hv)|_K(z) \\
			=  &h_z(\nabla\Pi_h^{\nabla}v )_1(z) - \frac{h_z}{L}( (\nabla\Pi_h^{\nabla}v )_1 +  \cdots  + (\nabla\Pi_h^{\nabla}v )_L)(z)  \\
			=  &\frac{h_z}{L}(((\nabla\Pi_h^{\nabla}v)_1-(\nabla\Pi_h^{\nabla}v)_2) +   \cdots  + ((\nabla\Pi_h^{\nabla}v )_1 - (\nabla\Pi_h^{\nabla}v )_L ))(z)
		\end{align*}
		and we only need to consider the term $h_z((\nabla\Pi_h^{\nabla}v )_1(z) - (\nabla\Pi_h^{\nabla}v )_2(z))$. By the inverse inequality of polynomials,
		\begin{align*}
			h_z((\nabla\Pi_h^{\nabla}v )_1(z) - (\nabla\Pi_h^{\nabla}v )_2(z))
			\lesssim & h_zh_e^{-1/2}|(\Pi_h^{\nabla}v )_1 - (\Pi_h^{\nabla}v )_2 |_{1,e}\nonumber\\
			\lesssim &  h_e^{1/2}(| (\Pi_h^{\nabla}v )_1 -v_1|_{1,e}+|v_2- (\Pi_h^{\nabla}v )_2 |_{1,e}+|[v]|_{1,e}).
		\end{align*}
		Using the trace inequality \eqref{trace} and inverse inequality \eqref{inverse_inequality}, we get
		\[
		| v  - \Pi_h^{\nabla}v  |_{1,e} \lesssim h_K^{ - 1/2}| v  - \Pi_h^{\nabla}v  |_{1,K}.
		\]
		For the jump term $|[v]|_{1,e}$, noting that $[\partial_{\bm t_e} v]$ is a polynomial on $e$, we obtain from the inverse inequality for polynomials to get $\|[\partial_{\bm t_e} v]\|_e \lesssim h_e^{-1} \|[v]\|_e$ and hence
		\[|[v]|_{1,e} \lesssim \|[\partial_{\bm n_e} v]\|_{0,e} + h_e^{-1} \|[v]\|_{0,e}.\]
		
		For a boundary vertex $z$, we have $\nabla E_h(v)(z) =\nabla v(z) = 0$. Let $e$ be the associated boundary edge. By the inverse inequality for polynomials and the trace inequality,
		\begin{align*}
			h_z|\nabla(\Pi_h^{\nabla}v-E_hv)|_K(z) |
			& = h_z|\nabla\Pi_h^{\nabla}v(z)| \le h_K\|\nabla\Pi_h^{\nabla}v\|_{\infty, e}
			\lesssim h_e^{-1/2} \|\Pi_h^{\nabla}v\|_{0, e} \\
			& = h_e^{-1/2} \|\Pi_h^{\nabla}v - v\|_{0, e}
			\lesssim | v  - \Pi_h^{\nabla}v |_{1,K} + h_K^{ - 1}\| v  - \Pi_h^{\nabla}v  \|_{0,K}.
		\end{align*}
		
		Therefore, for any vertex $z$ of $\mathcal{T}_h$, by using norm equivalence \eqref{NormEquivH1} and \eqref{NormEquivL2}, we have
		\begin{align*}
			\Big(h_K^{-1}h_z|\nabla(\Pi_h^{\nabla}v-E_hv)|_K(z)\Big)^2
			 \lesssim \sum_{e_i}\sum_{K\in \mathcal{T}_{e_i}} h_K^{-2}S^K(v-\Pi_h^\nabla v, v-\Pi_h^\nabla v)+\sum_{e_i}
			\frac{1}{|e_i|}\|[ \partial_{\bm n_{e_i}} v]\|_{0,e_i}^2 .
		\end{align*}

The above arguments yield the first estimate \eqref{C0}.

Step 2: Let $e$ be an edge shared by two elements $K^-$ and $K^+$. By the trace inequality \eqref{trace} and the inverse inequality \eqref{inverse_inequality}, one gets
\[\|[v]\|_{0,e} = \|[v-I_h^cv]\|_{0,e}
\lesssim \sum_{K=K^-,K^+} h_K^{-1/2} \|v-I_h^cv\|_{0,K}.\]
Noting that $\bm \chi(v)=\bm \chi(I_h^cv)$, we apply the norm equivalence \eqref{NormEquivL2} and the norm equivalence for $H^1$-conforming virtual elements (cf. \cite{Huang2021YuMedius}) to get
       \begin{align*}
           \|v-I_h^cv\|_{0,K}\leq \|v-\Pi_h^{\nabla}v\|_{0,K}+\|\Pi_h^{\nabla}v-I_h^cv\|_{0,K}\lesssim h_K\|\bm \chi(v-\Pi_h^{\nabla}v)\|_{\ell^2},
       \end{align*}
which yields the second estimate \eqref{jump0}.

Step 3: The third estimate follows by considering the triangle inequality
\begin{align}
    \Big\|\Big[\frac{\partial v}{\partial \bm n_e}    \Big]  \Big\|
    \le   \Big\|\Big[\frac{\partial \Pi_h^{\nabla }v}{\partial \bm n_e}    \Big]  \Big\|
    + \Big\|\Big[\frac{\partial (v-\Pi_h^{\nabla }v)}{\partial \bm n_e}    \Big]  \Big\|,
\end{align}
with the details omitted.
	\end{proof}
	
	\begin{remark}
		Our result aligns with the traditional finite element methods on a triangular mesh since the first term on the right-hand side of \eqref{C0} vanishes when $v$ is replaced by a polynomial of degree $\le k$  and the remaining parts are the jumps accounting for the discontinuities (see \cite{Brenner-Wang-Zhao-2004,Brenner2005SungC0IP,Carsten2023IP} for instance). We remark that the second estimate \eqref{jump0} implies that the interior penalty virtual element can be treated as a continuous element in the a posteriori estimate. This property is referred to as the quasi-$C^0$ continuity in this article.
	\end{remark}


	\section{A posteriori error estimation}\label{posteriori_analysis}
	
	For the a posteriori error estimation, without loss of generality, we simply assume homogeneous boundary value conditions.
	By reviewing the calculation in Section 5.1 of \cite{ChenHuangLin2022avem}, we can incorporate the computable components of the local error estimator $\eta_K$ as follows:
	\begin{align}\label{Local_Err_Estor}
		&\eta_{1,K} = \Big(\sum\limits_{e\in \mathcal{E}(K)\cap \mathcal{E}_h}  \eta_{1,e}^2\Big)^{1/2}, \qquad
		\eta_{1,e} =\frac{1}{|e|^{1/2}}\Big\| \Big[ \frac{\partial \Pi_h^\nabla u_h}{\partial \bm n_e} \Big]\Big\|_{0,e},
		\nonumber \\
		& \eta_{2,K} = \Big(\sum\limits_{e\in \mathcal{E}(K)\cap \mathcal{E}_h^0}  \eta_{2,e}^2\Big)^{1/2}, \qquad
		\eta_{2,e}=|e|^{1/2}\Big\| \Big[ \frac{\partial^2\Pi_h^\nabla u_h}{\partial \bm n_e^2} \Big]\Big\|_{0,e},
		\nonumber \\
		&\eta_{3,K} = \Big(\sum\limits_{e\in \mathcal{E}(K)\cap \mathcal{E}_h^0}  \eta_{3,e}^2\Big)^{1/2}, \qquad
		\eta_{3,e}=|e|^{3/2}\Big\| \Big[ \frac{\partial (\Delta \Pi_h^\nabla u_h)}{\partial \bm n_e} + \frac{\partial^3 ( \Pi_h^\nabla u_h)}{\partial \bm{n}_e\partial\bm t_e^2} \Big]\Big\|_{0,e},
		\nonumber \\
		& \eta_{4,K}= h_K^{-1} \big( S^K(u_h-\Pi_h^\nabla u_h,u_h-\Pi_h^\nabla u_h)\big)^{1/2},\nonumber\\
		&\eta_{5,K}:=h_K^2\|f - f_h\|_{0,K},
		\nonumber\\
		&\eta_{6,K}=h_K^2\|f_h-\Delta^2\Pi_h^\nabla u_h\|_{0,K},
	\end{align}
	where $f_h|_K = \Pi_{0,K}^k f$ and $\omega(K)$ denotes the union of all elements in $\mathcal{T}_h$ sharing some edges or points with $K$. The first term results from the penalty term which measures the extent to which $u_h$ fails to be in $H_0^2(\Omega)$. The remaining five terms successively indicate the residual of the normal moment jump and the jump in the effective shear force along interior edges, the virtual element consistency residual,  the element data oscillation and the element residual.
	
	The local and global error estimators are respectively defined by
	\[\eta_K(u_h) = \Big( \sum\limits_{i=1}^6 \eta_{i,K}^2\Big)^{1/2}, \quad \eta(u_h) = \eta(u_h, \mathcal{T}_h) = \Big( \sum\limits_{K\in \mathcal{T}_h} \eta_K^2(u_h) \Big)^{1/2}.\]
Note that we do not include $\eta_{0,K}$ in the error estimator, where $\eta_{0,K}$ is defined in \eqref{eta0}.
	
	\subsection{Upper bound}
	
	\begin{theorem}\label{upper bound}
		Let $u \in H_0^2(\Omega) \cap H^{k+1}(\Omega) $ with $k\ge 2$. Suppose that $u$ and $u_h$ are the solutions to \eqref{strongform} and \eqref{IPVEM}, respectively. Then it holds
		\[\|u - u_h\|_h \lesssim \eta(u_h).\]
	\end{theorem}
	\begin{proof}
		Let $e_h = u -u_h$ and $E_h$ be the enriching operator. Using the triangle inequality and \eqref{C0}, one has
		\begin{align}
			\|e_h\|_h
			& \le |u - E_h u_h|_2 + |E_h u_h - u_h|_{2,h} + J_1(e_h,e_h)^{1/2} \label{uppereh}\\
			& \lesssim |u - E_h u_h|_2 + \Big(\sum_{K\in \mathcal{T}_h } S^K(u_h-\Pi_h^\nabla u_h, u_h-\Pi_h^\nabla u_h) \nonumber\\
			& \qquad + \sum\limits_{e\in \mathcal{E}_h} \frac{1}{|e|} \Big\| \Big[\frac{\partial \Pi_h^\nabla u_h}{\partial \bb{n}_e}\Big]\Big\|_{0,e}^2 \Big)^{1/2}  + J_1(e_h,e_h)^{1/2} \nonumber \\
			& \lesssim |u - E_h u_h|_2 +\Big(\sum_{K\in \mathcal{T}_h}( \eta_{1,K}^2 + \eta_{4,K}^2)\Big)^{1/2}. \nonumber
		\end{align}
		Therefore, we only need to consider the first term $|u - E_h u_h|_2$ in the last equation.
		
		Let $\phi = u - E_h u_h$. According to the definitions of the continuous and discrete variational problems, one gets
		\begin{align*}
			& |u - E_h u_h|_2^2 =  a(u-E_hu_h,\phi) \\
			& =\sum_{K\in\mathcal{T}_h} \Big( a^K(u_h-E_hu_h, \phi) - a^K(u_h,\phi-I_h\phi)\Big)
			+ a(u,\phi) -\sum_{K\in \mathcal{T}_h}a^K(u_h, I_h\phi) \\
			&  = \sum_{K\in\mathcal{T}_h} \Big( a^K(u_h-E_hu_h, \phi) - a^K(u_h,\phi-I_h\phi)\Big) + (f,\phi) - (f_h, I_h \phi)  \\
			& \qquad  + a_h(u_h,I_h \phi)  -\sum_{K\in \mathcal{T}_h}a^K(u_h, I_h\phi)
			+J_1(u_h,I_h\phi)+J_2(u_h,I_h\phi)+J_3(u_h,I_h\phi).
		\end{align*}
		The definition of $\Pi_h^{\Delta}$ gives
		\begin{align*}
			&a_h(u_h,I_h\phi)-\sum_{K\in \mathcal{T}_h}a^K(u_h,I_h\phi)\notag\\
			=&\sum_{K\in \mathcal{T}_h} a^K(\Pi_h^{\Delta}u_h,\Pi_h^{\Delta}I_h\phi)+\sum_{K\in \mathcal{T}_h} S^K(u_h-\Pi_h^{\Delta}u_h,I_h\phi-\Pi_h^{\Delta}I_h\phi)-\sum_{K\in \mathcal{T}_h}a^K(u_h,I_h\phi)\notag\\
			=&\sum_{K\in \mathcal{T}_h} a^K(\Pi_h^{\Delta}u_h-u_h,I_h\phi)+\sum_{K\in \mathcal{T}_h} S^K(u_h-\Pi_h^{\Delta}u_h,I_h\phi-\Pi_h^{\Delta}I_h\phi)
		\end{align*}
		Therefore, we can decompose $|u - E_h u_h|_{2,h}^2$ as
		\begin{align}\label{sum}
			a(u-E_hu_h,\phi)=I_1 + I_2 +I_3 + I_4 +J_1(u_h,I_h\phi)+J_2(u_h,I_h\phi)+J_3(u_h,I_h\phi),
		\end{align}
		where
		\begin{align}
			&I_1= \sum_{K\in\mathcal{T}_h} a^K(u_h-E_hu_h, \phi),\nonumber \\
			&I_2=(f_h,\phi-I_h\phi) -\sum_{K\in\mathcal{T}_h}a^K(\Pi_h^{\Delta}u_h,\phi-I_h\phi) \label{upperI2} \\
			&I_3= -\sum_{K\in \mathcal{T}_h} a^K(u_h-\Pi_h^{\Delta}u_h,\phi)+\sum_{K\in \mathcal{T}_h}  \nonumber S^K(u_h-\Pi_h^{\Delta}u_h,I_h\phi-\Pi_h^{\Delta}I_h\phi) \nonumber\\
			&I_4= (f- f_h,\phi). \nonumber
		\end{align}
		
		For the term $I_1$, using the Cauchy-Schwarz inequality and the estimate in Lemma \ref{lem:enriching}, one gets
		\begin{align*}
			I_1
			&  \le \sum_{K\in \mathcal{T}_h} |u_h-E_hu_h|_{2,K}|\phi|_{2,K}
			\le \Big(\sum_{K\in \mathcal{T}_h} |u_h-E_hu_h|_{2,K}^2\Big)^{1/2}|\phi|_2 \\
			&
			\lesssim \Big(\sum_{K\in \mathcal{T}_h}( \eta_{1,K}^2 + \eta_{4,K}^2) \Big)^{1/2} |\phi|_2.
		\end{align*}

		For the term $I_2$, using the integration by parts we can split it into several terms as
\begin{align*}
I_2
& = \sum_{K\in\mathcal{T}_h}a^K(\Pi_h^{\nabla}u_h - \Pi_h^{\Delta}u_h,\phi-I_h\phi) + (f_h,\phi-I_h\phi) -\sum_{K\in\mathcal{T}_h}a^K(\Pi_h^{\nabla}u_h,\phi-I_h\phi) \\
& = \sum_{K\in\mathcal{T}_h}a^K(\Pi_h^{\nabla}u_h - \Pi_h^{\Delta}u_h,\phi-I_h\phi) \\
& \qquad + \sum_{K\in \mathcal{T}_h}  (f_h - \Delta^2\Pi_h^{\nabla}u_h,\phi-I_h\phi)_K \\
			& \qquad + \sum_{K\in \mathcal{T}_h} \Big(\frac{\partial (\Delta \Pi_h^{\nabla}u_h)}{\partial \bm n}+\frac{\partial^3 \Pi_h^{\nabla}u_h}{\partial \bm n \partial \bm t^2}, \phi-I_h\phi\Big)_{\partial K} \\
			& \qquad - \sum_{K\in \mathcal{T}_h} \Big(\frac{\partial^2\Pi_h^{\nabla}u_h}{\partial \bm n^2}, \frac{\partial(\phi-I_h\phi)}{\partial \bm n}\Big)_{\partial K} \\
			& \qquad + \sum_{K\in \mathcal{T}_h} \sum\limits_{z \in \mathcal{V}(K)} \Big[\frac{\partial^2\Pi_h^{\nabla}u_h }{\partial \bm t \partial\bm{n}} \Big](z)(\phi-I_h\phi)(z) \\
			& =: I_{2,0} + I_{2,1} + I_{2,2} + I_{2,3} + I_{2,4}.
		\end{align*}
For $I_{2,0}$, we obtain from the boundedness of $\Pi_h^\Delta$ in Lemma \ref{lem:bound}, the interpolation error estimate \eqref{interp}, the inverse inequality \eqref{inverse_inequality} and the norm equivalence \eqref{NormEquivH1} that
\begin{align*}
I_{2,0}
& = \sum_{K\in\mathcal{T}_h}a^K(\Pi_h^{\nabla}u_h - \Pi_h^{\Delta}u_h,\phi-I_h\phi)
  \le \sum_{K\in\mathcal{T}_h}|\Pi_h^{\nabla}u_h - \Pi_h^{\Delta}u_h|_{2,K} |\phi-I_h\phi|_{2,K} \\
& \lesssim \sum_{K\in\mathcal{T}_h}|u_h - \Pi_h^{\nabla}u_h|_{2,K} |\phi|_{2,K}
\lesssim \sum_{K\in\mathcal{T}_h} h_K^{-1} |u_h - \Pi_h^{\nabla}u_h|_{1,K} |\phi|_{2,K} \\
& \lesssim \Big( \sum_{K\in \mathcal{T}_h} \eta_{4,K}^2 \Big)^{1/2} |\phi|_2.
\end{align*}
For $I_{2,1}$, by the interpolation error estimate \eqref{interp},
		\begin{align*}
			I_{2,1}
			& = \sum_{K\in \mathcal{T}_h}  (f_h - \Delta^2\Pi_h^{\nabla}u_h,\phi-I_h\phi)_K
			\le \sum_{K\in \mathcal{T}_h}  \|f_h - \Delta^2\Pi_h^{\nabla}u_h\|_{0,K} \|\phi-I_h\phi\|_{0,K}\\
			& \lesssim \sum_{K\in \mathcal{T}_h}  h_K^2  \|f_h - \Delta^2\Pi_h^{\nabla}u_h\|_{0,K} |\phi|_{2,K}
			\le \Big( \sum_{K\in \mathcal{T}_h} \eta_{6,K}^2 \Big)^{1/2} |\phi|_2.
		\end{align*}
		For the sake of brevity, we set $Q_3(u_h) = \frac{\partial (\Delta \Pi_h^{\nabla}u_h)}{\partial \bm n}+\frac{\partial^3 \Pi_h^{\nabla}u_h}{\partial \bm n \partial \bm t^2}$ in $I_{2,2}$. From Eq.~\eqref{magic}  one gets
		\begin{align*}
			I_{2,2} =& \sum_{K\in \mathcal{T}_h} (Q_3(u_h), \phi-I_h\phi)_{\partial K}\\
			= &\sum\limits_{e \in \mathcal{E}_h^0}  ([Q_3(u_h)], \{\phi-I_h\phi\})_e + \sum\limits_{e \in \mathcal{E}_h}  (\{Q_3(u_h)\}, [\phi-I_h\phi])_e.
		\end{align*}
		On the other hand, noting that $I_h\phi|_e\in \mathbb{P}_{k+2}(e)$ and $Q_3(u_h)|_e \in \mathbb{P}_{k-3}(e)$, one has $(\{Q_3(u_h)\}, [I_h\phi])_e = 0$, since $I_h\phi$ is continuous at the Gauss-Lobatto points on $e$ and the quadrature formula is exact for polynomials of up to degree $2k-1$. Combing with $\phi\in H_0^2(\Omega)$, we derive
		\begin{equation*}
			\sum\limits_{e \in \mathcal{E}_h}  (\{Q_3(u_h)\}, [\phi-I_h\phi])_e=0.
		\end{equation*}
		The continuity of $I_h\phi$ at the Gauss-Lobatto points on $e$  implies that
		\[
		(\phi - I_h \phi)(z) = 0,\quad z\in \mathcal{V}_h,
		\]
		which indicates
		\[I_{2,4} = \sum_{K\in \mathcal{T}_h} \sum\limits_{z \in \mathcal{V}(K)} \Big[\frac{\partial^2\Pi_h^{\nabla}u_h }{\partial \bm t \partial\bm{n}} \Big](z)(\phi-I_h\phi)(z)= 0.\]
		Hence, it follows from the Cauchy-Schwarz inequality, the trace inequality \eqref{trace} and the interpolation error estimates \eqref{interp} that
		\begin{align*}
			I_{2,2}
			& = \sum\limits_{e \in \mathcal{E}_h^0}  ([Q_3(u_h)], \{\phi-I_h\phi\})_e
			\le \sum\limits_{e \in \mathcal{E}_h^0}\|[Q_3(u_h)]\|_{0,e} \|\{\phi-I_h\phi\}\|_{0,e}  \\
			& \lesssim \sum\limits_{e \in \mathcal{E}_h^0}\|[Q_3(u_h)]\|_{0,e} ( h_e^{-1/2}  \|\{\phi-I_h\phi\}\|_{0,\omega(e)} + h_e^{1/2} |\{\phi-I_h\phi\} |_{1,\omega(e)} ) \\
			& \lesssim \sum\limits_{e \in \mathcal{E}_h^0} h_e^{3/2} \|[Q_3(u_h)]\|_{0,e}  |\phi|_{2,\omega(e)}
			\lesssim \Big( \sum_{e\in \mathcal{E}_h^0} \eta_{3,e}^2 \Big)^{1/2} |\phi|_2
			\lesssim \Big( \sum_{K\in \mathcal{T}_h} \eta_{3,K}^2 \Big)^{1/2} |\phi|_2.
		\end{align*}
		For $I_{2,3}$, noting that $[\frac{\partial \phi}{\partial \bm n}]_e = 0$, from \eqref{magic} and the definition of $V_k(K)$, we have
		\begin{align}
			I_{2,3}
			& = - \sum_{K\in \mathcal{T}_h} \Big(\frac{\partial^2\Pi_h^{\nabla}u_h}{\partial \bm n^2}, \frac{\partial(\phi-I_h\phi)}{\partial \bm n}\Big)_{\partial K} \nonumber \\
			& = - \sum_{e\in \mathcal{E}_h^0} \Big( ([\frac{\partial^2\Pi_h^{\nabla}u_h}{\partial \bm n^2}], \{\frac{\partial(\phi-I_h\phi)}{\partial \bm n}\})_e
			- \sum_{e\in \mathcal{E}_h} (\{\frac{\partial^2\Pi_h^{\nabla}u_h}{\partial \bm n^2}\}, [\frac{\partial(\phi-I_h\phi)}{\partial \bm n}])_e \Big)\nonumber \\
			& = - \sum_{e\in \mathcal{E}_h^0} ([\frac{\partial^2\Pi_h^{\nabla}u_h}{\partial \bm n^2}], \{\frac{\partial(\phi-I_h\phi)}{\partial \bm n}\})_e
			- J_2(u_h, I_h \phi), \label{I23}
		\end{align}
which gives
\[I_{2,3} + J_2(u_h, I_h \phi) = - \sum_{e\in \mathcal{E}_h^0} ([\frac{\partial^2\Pi_h^{\nabla}u_h}{\partial \bm n^2}], \{\frac{\partial(\phi-I_h\phi)}{\partial \bm n}\})_e.\]
		Similarly, by using the Cauchy-Schwarz inequality, the trace inequality \eqref{trace} and the interpolation error estimates \eqref{interp}, we can get
		\[I_{2,3} + J_2(u_h, I_h \phi)\lesssim \Big( \sum_{K\in \mathcal{T}_h} \eta_{2,K}^2 \Big)^{1/2} |\phi|_2 . \]
		Collecting the above estimates to derive
		\[I_2 + J_2(u_h, I_h \phi) \lesssim \Big( \sum_{K\in \mathcal{T}_h} (\eta_{2,K}^2 + \eta_{3,K}^2 + \eta_{4,K}^2 + \eta_{6,K}^2) \Big)^{1/2} |\phi|_2.\]
		
		For the term $I_3$, by the norm equivalence \eqref{Norm_Equiv_Qk_H2K} and \eqref{NormEquivH1}, the boundedness of $\Pi_h^{\Delta}$ in Lemma \ref{lem:bound}, the inverse inequality \eqref{inverse_inequality}, we obtain
		\begin{align*}
			I_3
			& = -\sum_{K\in \mathcal{T}_h} a^K(u_h-\Pi_h^{\Delta}u_h,\phi)+\sum_{K\in \mathcal{T}_h} S^K(u_h-\Pi_h^{\Delta}u_h,I_h\phi-\Pi_h^{\Delta}I_h\phi) \\
			& \lesssim \sum_{K\in\mathcal{T}_h}\Big(
			h_K^{-2} \| {\boldsymbol{\chi}}(u_h-\Pi_h^{\Delta}u_h) \|_{\ell^2}
			\| {\boldsymbol{\chi}}({I_h}\phi-\Pi_h^{\Delta}({I_h}\phi)) \|_{\ell^2}
			+ | u_h-\Pi_h^{\Delta}u_h |_{2,K} | \phi |_{2,K} \Big)  \\
			& \lesssim \sum_{K\in\mathcal{T}_h}\Big(
			| u_h-\Pi_h^{\Delta}u_h |_{2,K} | {I_h}\phi-\Pi_h^{\Delta}({I_h}\phi) |_{2,K}
			+ | u_h-\Pi_h^{\Delta}u_h |_{2,K} | \phi|_{2,K} \Big)  \\
			& \lesssim \Big(\sum_{K\in\mathcal{T}_h}
			| u_h-\Pi_h^{\Delta}u_h |_{2,K}^2\Big)^{1/2}
			\Big( \sum_{K\in\mathcal{T}_h}
			\big( |\phi|_{2,K}^2 + | \phi-{I_h}\phi |_{2,K}^2 \big)\Big)^{1/2}.  \\
            & \lesssim \Big(\sum_{K\in\mathcal{T}_h}
			| u_h-\Pi_h^{\nabla}u_h |_{2,K}^2+|\Pi_h^{\nabla}u_h-\Pi_h^{\Delta}u_h|_{2,K}\Big)^{1/2}
			|\phi |_2.  \\
             & \lesssim \Big(\sum_{K\in\mathcal{T}_h}
			h_K^{-2}| u_h-\Pi_h^{\nabla}u_h |_{1,K}^2\Big)^{1/2}
			|\phi |_2
			\lesssim \Big(\sum_{K\in \mathcal{T}_h} \eta_{4,K}^2(u_h) \Big)^{1/2}|\phi|_2.
		\end{align*}

		For $I_4$, it is easy to get
		\[I_4= (f- f_h,\phi) \lesssim \Big(\sum_{K\in \mathcal{T}_h} \eta_{5,K}^2(u_h) \Big)^{1/2}|\phi|_2.\]
		
		For $J_1(u_h,I_h\phi)$, by the trace inequality \eqref{trace}, the boundedness of $\Pi_h^{\nabla}$ and  interpolation error estimate \eqref{interp}, we have
		\begin{align*}
		J_1(u_h,I_h\phi)
		&\lesssim(\sum_{e\in\mathcal{E}_h} \eta_{1,e}^2)^{1/2}\Big(\sum_{e\in\mathcal{E}_h}|e|^{-1}\Big\|\Big[ \frac{\partial(\Pi_h^{\nabla}I_h\phi-\phi)}{\partial \bm n_e}\Big]\Big\|_{0,e}^2\Big)^{1/2}
		\lesssim \Big(\sum_{K\in \mathcal{T}_h}\eta_{1,K}^2\Big)^{1/2}|\phi|_2.
		\end{align*}
		For $J_3(u_h,I_h\phi)$, by the Cauchy-Schwarz inequality, trace inequality \eqref{trace}, the inverse inequalities of polynomials and  the boundedness of $\Pi_h^{\nabla}$, ones give
		\begin{align}
			J_3(u_h,I_h\phi)
			& =	\Big|\sum_{e\in \mathcal{E}_h}\int_e\Big\{ \frac{\partial^2\Pi_h^{\nabla}I_h\phi}{\partial \bm n_e^2} \Big\}\Big[\frac{\partial \Pi_h^{\nabla} u_h}{\partial \bm n_e} \Big]\d s\Big|\lesssim (\sum_{e\in\mathcal{E}_h} \eta_{1,e}^2)^{1/2}\Big(\sum_{K\in \mathcal{T}_h}|\Pi_h^{\nabla}I_h\phi|_{2,K}^2\Big)^{1/2} \nonumber \\
			&\lesssim \Big(\sum_{K\in \mathcal{T}_h}\eta_{1,K}^2\Big)^{1/2}|\phi|_2. \label{upperJ3}
		\end{align}
		
		The combination of the above estimates yields
		\[|u - E_h u_h|_2^2 =  a(u-E_hu_h,\phi) \lesssim \eta(u_h) |\phi|_2 = \eta(u_h) |u - E_h u_h|_2, \]
		as required.
	\end{proof}

	\subsection{Lower bound}

	By means of the technique of bubble functions frequently used in the a posteriori error analysis \cite{Huang2021YuMedius,Verfurth2013}, we
	can also derive the following lower bound of the a posteriori error estimator.  The proof  is similar to \cite{Huang2021YuMedius}, but for the sake of completeness, we  still provide a detailed proof.
	
	\begin{theorem}\label{posteriori}
		Let $u\in V$ be the solution of \eqref{origion}. Then, for any  $w_h\in V_h$, $K\in \mathcal{T}_h$ and  $e\in \mathcal{E}(K)\cap \mathcal{E}_h^0$, there hold
		\begin{align}
			h^4_K\| f_h-\Delta^2 \Pi_h^{\nabla}w_h \|_{0,K}^2
			&\lesssim | u-\Pi_h^{\nabla}w_h |_{2,K}^2 + \eta_{5,K}^2,
			\label{Local_Eff_Est1}  \\
			h_e\Big\|\Big[\frac{\partial^2 \Pi_h^{\nabla}w_h}{\partial\bm n_e^2}\Big]\Big\|_{0,e}^2
			&\lesssim | u-\Pi_h^{\nabla}w_h |_{2,\omega(e)}^2
			+ \sum_{K^{'}\in\omega(e)}\eta_{5,K^{'}}^2,
			\label{Local_Eff_Est2} \\
			h_e^3\Big\| \Big[ \frac{\partial \Pi_h^{\nabla}w_h}{\partial \bm n_e} + \frac{\partial^3  \Pi_h^{\nabla}w_h}{\partial \bm{n}_e\partial\bm t_e^2} \Big]\Big\|_{0,e}^2
			&\lesssim | u-\Pi_h^{\nabla}w_h |_{2,\omega(e)}^2
			+ \sum_{K^{'}\in\omega(e)}\eta_{5,K^{'}}^2.
			\label{Local_Eff_Est3}
		\end{align}
	\end{theorem}
	
	\begin{proof}
		We divide the proof into three steps. For brevity, we omit the subscript $e$ in $\bm t_e$ and $\bm{n}_e$. For the sake of brevity, we denote $v=\Pi_h^{\nabla}w_h$ in what follows.
		
		Step 1: For any $K\in\mathcal{T}_h$, assume $\mathcal{T}_K$ is the auxiliary triangulation of $K$,
		and let $b_K=\sum_{\tau\in\mathcal{T}_K}(\lambda_{1,\tau}\lambda_{2,\tau}\lambda_{3,\tau})^2$
		be the element bubble function defined on $K$, where $\lambda_{i,\tau}$($i=1,2,3$)
		are the nodal basis functions at vertices of virtual triangles $\tau\in\mathcal{T}_K$.
		Set the polynomial $\phi|_K=b_K(f_h-\Delta^2  v )|_K$
		and extend it to be zero in $\Omega\backslash{K}$.
		Denote the auxiliary function as $\phi$, and it is obviously that $\phi\in H^2_0(\Omega)$ and
		$\partial_{\boldsymbol{n}}\phi|_{\partial{K}}=\phi|_{\partial{K}}=0$, which yields
		\begin{align}
			(f-\Delta^2v,\,\phi)
			&= a(u,\,\phi) - (\Delta^2v,\,\phi)_K=a^K(u-v,\phi). \nonumber
		\end{align}
		Then, using the inverse inequalities of polynomials and the technique of bubble functions, there holds
		\begin{align}
			\|f_h - \Delta^2 v\|_{0,K}^2
			&\lesssim (f_h-\Delta^2 v,\,\phi)_K
			= (f_h-f,\,\phi)_{K} + (f-\Delta^2 v,\,\phi)_K \nonumber\\
			&= (f_h-f,\,\phi)_{K} + a^K(u-v,\,\phi)\nonumber\\
			&\leq \|f-f_h\|_{0,K}\|\phi\|_{0,K} + |u-v|_{2,K}|\phi|_{2,K}\nonumber\\
			& \lesssim (\|f-f_h\|_{0,K}+h_K^{-2}|u-v|_{2,K})\|\phi\|_{0,K}\nonumber\\
			& \lesssim (\|f-f_h\|_{0,K}+h_K^{-2}|u-v|_{2,K})\|f_h
			- \Delta^2 v\|_{0,K}.\nonumber
		\end{align}
		The required estimate \eqref{Local_Eff_Est1} is obtained easily.
		
		Step 2. For any interior edge $e\in\mathcal{E}_h^0$ shared by two elements $K^{+}$ and $K^{-}$,
		assume that $\tau^{\pm}$ are two triangles sharing the edge $e$ in the auxiliary triangulation
		$\mathcal{T}_{K^{+}}\cup\mathcal{T}_{K^{-}}$, viz. $e=\tau^{+}\cap\tau^{-}$,
		where $\tau^{+}\in\mathcal{T}_{K^{+}}$ and $\tau^{-}\in\mathcal{T}_{K^{-}}$.
		Denote $T=\tau^{+}\cup\tau^{-}$. Let $\lambda_{1,e}^{\pm}$ and $\lambda_{2,e}^{\pm}$ be
		the corresponding nodal basis functions of two endpoints of $e$ in $\tau^{\pm}$, and
		denote the unit outward normal vector of $\tau^{+}$ and $\tau^{-}$ along $e$
		by $\boldsymbol{n}^{+}$ and $\boldsymbol{n}^{-}$, respectively.
		
		According to the technique of bubble functions,
		we need to construct an auxiliary polynomial $\phi\in H^2_0(T)$ satisfying
		$\partial_{\boldsymbol{n}}\phi|_{ \partial{T}\setminus{e}}=0$ and
		\begin{eqnarray}
			\label{Mnn_Jump_L2e_Est}
			\Big	\|\Big[\frac{\partial^2 v}{\partial\bm n^2}\Big]\Big\|^2_{0,e}\lesssim \Big(	\Big[\frac{\partial^2 v}{\partial\bm n^2}\Big],\partial_{\bm n}\phi\Big)_e.
		\end{eqnarray}
		
		Since the jump $[\frac{\partial^2 v}{\partial\bm n^2}]|_e$ is a polynomial on $e$,
		we extend it to $T$, denoted by $E_{T}([\frac{\partial^2 v}{\partial\bm n^2}])$,
		so that this function is constant along the lines perpendicular to $e$.
		Let $\phi=b_eE_T([\frac{\partial^2 v}{\partial\bm n^2}])l_e$,
		where $l_e$ is the function of line $e$, and $b_e$ is edge bubble function of $e$ defined by
		$b_e=(\lambda_{1,e}^{+}\lambda_{2,e}^{+}\lambda_{1,e}^{-}\lambda_{2,e}^{-})^2$.
		A simple manipulation yields
		$\partial_{\boldsymbol{n}}\phi|_{e}=b_e[\frac{\partial^2 v}{\partial\bm n^2}]|_e$ and $\phi|_e = 0$. 	
		Next, by \eqref{Mnn_Jump_L2e_Est},
		there holds
		\begin{align}
			\Big	\|\Big[\frac{\partial^2 v}{\partial\bm n^2}\Big]\Big\|^2_{0,e}\lesssim \Big(	\Big[\frac{\partial^2 v}{\partial\bm n^2}\Big],\partial_{\bm n}\phi\Big)=a^T(v,\,\phi) - (\Delta^2v,\,\phi)_T.
			\label{Local_Eff_Est2_aux1}
		\end{align}
		
		Since $\phi\in H^2_0(T)$ can be viewed as a function in $H_0^2(\Omega)$ by zero extension,
		we have
		\begin{align}
			\label{Local_Eff_Est2_aux2}
			a^T(u,\, \phi)=(f,\phi)_T.
		\end{align}
		Using \eqref{Local_Eff_Est2_aux1}, \eqref{Local_Eff_Est2_aux2}, the Cauchy-Schwarz inequality and
		the inverse inequalities of polynomials, we find
		\begin{align}
			\Big	\|\Big[\frac{\partial^2 v}{\partial\bm n^2}\Big]\Big\|^2_{0,e}
			&\lesssim 	a^T(v-u,\, \phi) + (f-\Delta^2v,\,\phi)_T \nonumber\\
			&\lesssim\sum_{\tau=\tau^{+},\tau^{-}}\Big(|v-u|_{2,\tau}|\phi|_{2,\tau}
			+ \|f-\Delta^2 v\|_{0,\tau}\|\phi\|_{0,\tau}\Big) \nonumber\\
			&\lesssim\sum_{\tau=\tau^{+},\tau^{-}}\Big(h_{\tau}^{-2}|v-u|_{2,\tau}
			+ \|f-\Delta^2 v\|_{0,\tau}\Big)\|\phi\|_{0,\tau}.
			\label{Local_Eff_Est2_aux3}
		\end{align}
		On the other hand, the inverse inequalities of polynomials imply
		\begin{align}
			\|\phi\|_{0,\tau^{\pm}}
			&=h_{\tau^{\pm}}\Big\|b_eE_T\Big(\Big[\frac{\partial^2 v}{\partial\bm n^2}\Big]\Big|_e\Big)
			\frac{l_e}{h_{\tau^{\pm}}}\Big\|_{0,\tau^{\pm}}
			\leq h_{\tau^{\pm}}\Big\|E_T\Big(\Big[\frac{\partial^2 v}{\partial\bm n^2}\Big]\Big|_e\Big)\Big\|_{0,\tau^{\pm}}\nonumber\\
			&\lesssim h_{\tau^{\pm}}^2\Big\|E_T\Big(\Big[\frac{\partial^2 v}{\partial\bm n^2}\Big]\Big|_e\Big)\Big\|_{\infty,\tau^{\pm}}
			\lesssim h_{\tau^{\pm}}^2\Big\|\Big[\frac{\partial^2 v}{\partial\bm n^2}\Big]\Big\|_{\infty,e}\nonumber\\
			&\lesssim h_{\tau^{\pm}}^{3/2}\Big\|\Big[\frac{\partial^2 v}{\partial\bm n^2}\Big]\Big\|_{0,e}
			\lesssim h_{K^{\pm}}^{3/2}\Big\|\Big[\frac{\partial^2 v}{\partial\bm n^2}\Big]\Big\|_{0,e}.
			\label{Local_Eff_Est2_aux4}
		\end{align}
		Hence,
		\begin{align}
			h_e \Big\|\Big[\frac{\partial^2 v}{\partial\bm n^2}\Big]\Big\|_{0,e}^2
			&\lesssim\sum_{\tau=\tau^{+},\tau^{-}}\Big(|v-u|_{2,\tau}^2
			+ h_{\tau}^{4}\|f-\Delta^2 v\|_{0,\tau}^2\Big) \nonumber\\
			&\lesssim |v-u|_{2,\omega(e)}^2 + \sum_{K^{'}\in\omega(e)}\eta_{6,K^{'}}^2.
			\nonumber
		\end{align}
		
		Step 3. Analogously, we need to construct an edge bubble function $b_e$ and
		an auxiliary polynomial $\phi\in H^2(T)$ such that
		\[
		\phi|_e=b_e\Big[ \frac{\partial \Delta v}{\partial \bm n} + \frac{\partial^3  v}{\partial \bm{n}\partial\bm t^2} \Big]\Big|_e
		\quad \text{and} \quad
		\phi|_{\overline{\partial T\setminus e}}
		=\partial_{\boldsymbol{n}}\phi|_{\overline{\partial T\setminus e}}=0,
		\]
		where $\overline{\partial T\setminus e}$ is the closure of $\partial T\setminus e$.
		Divide the triangle ${\tau}^{-}$ (resp. ${\tau}^{+}$)
		into two small triangles ${\tau}_1^{-},{\tau}_2^{-}$ (resp. ${\tau}_1^{+}, {\tau}_2^{+}$)
		by connecting the midpoint of $e$ (denoted by $a_e$)
		and the vertex of ${\tau}^{-}$ (resp. ${\tau}^{+}$) opposite to the edge $e$.
		Assume the barycenter coordinate functions of the two small triangles at point $a_e$
		are $b_1^{-}$, $b_2^{-}$ (resp. $b_1^{+}$, $b_2^{+}$).
		Let $b_e=(b_1^{-}b_2^{-}b_1^{+}b_2^{+})^{2}$ and
		$\phi=b_eE_T\Big(\Big[ \frac{\partial \Delta v}{\partial \bm n} + \frac{\partial^3  v}{\partial \bm{n}\partial\bm t^2} \Big]\Big)$.
		Then, we can check that these two functions satisfy the desired conditions.
		Furthermore, we have
		\begin{align}
			\Big\|\Big[ \frac{\partial \Delta v}{\partial \bm n} + \frac{\partial^3  v}{\partial \bm{n}\partial\bm t^2} \Big]\Big\|_{0,e}^2
			\lesssim \Big(\Big[ \frac{\partial \Delta v}{\partial \bm n} + \frac{\partial^3  v}{\partial \bm{n}\partial\bm t^2} \Big],\phi\Big)_e,
			\nonumber
		\end{align}
		and by Green's identity,
		there holds
		\small{\begin{align}
				\Big\| \Big[ \frac{\partial \Delta v}{\partial \bm n} + \frac{\partial^3  v}{\partial \bm{n}\partial\bm t^2} \Big] \Big\|_{0,e}^2
				&\lesssim \sum_{\tau=\tau^-,\tau^+} \Big((\Delta^2 v, \, \phi)_{\tau} - a^{\tau} (v, \, \phi)
				+ \Big(\Big[\frac{\partial^2 v}{\partial\bm n^2}\Big], \,\partial_{\boldsymbol{n}^{\tau}}\phi\Big)_{0,e}\Big),
				\label{MnttQn_Jump_L2e_Est}
		\end{align}}
		where $\boldsymbol{n}^{\tau}$ is the unit outward normal vector of $\tau$ along $e$.
		
		Since $\phi\in H_0^2(T)$, the equation \eqref{Local_Eff_Est2_aux2} holds.
		Then, using  \eqref{Local_Eff_Est2_aux2}, \eqref{MnttQn_Jump_L2e_Est},
		the Cauchy-Schwarz inequality and the inverse inequalities of polynomials, we have
		\begin{align}
			\Big\| \Big[& \frac{\partial \Delta v}{\partial \bm n} + \frac{\partial^3  v}{\partial \bm{n}\partial\bm t^2} \Big] \Big\|_{0,e}^2
			\lesssim\sum\limits_{\tau=\tau^{+},\tau^{-}}
			\Big((\Delta^2v-f,\,\phi)_{\tau} + a^{\tau} (u-v, \, \phi) \Big)
			+ \Big(\Big[\frac{\partial^2 v}{\partial\bm n^2}\Big],\partial_{\boldsymbol{n}}\phi\Big)_{0,e}\nonumber\\
			&\lesssim\sum_{\tau=\tau^{+},\tau^{-}}\Big( \|f-\Delta^2v\|_{0,\tau}
			+ h_{\tau}^{-2}|v-u|_{2,\tau}\Big)\|\phi\|_{0,\tau}
			+ \Big\|\Big[\frac{\partial^2 v}{\partial\bm n^2}\Big]\Big\|_{0,e}\|\partial_{\boldsymbol{n}}\phi\|_{0,e}
			\label{Local_Eff_Est3_aux1}.
		\end{align}	
		Similar to \eqref{Local_Eff_Est2_aux4}, one has
		\small{\begin{align}
				\|\phi\|_{0,\tau^{\pm}}\lesssim \Big\|E_T\Big(\Big[ \frac{\partial \Delta v}{\partial \bm n} + \frac{\partial^3  v}{\partial \bm{n}\partial\bm t^2} \Big] \Big)\Big\|_{0,\tau^{\pm}}
				\lesssim h_{K^{\pm}}^{1/2}\Big\|\Big[ \frac{\partial \Delta v}{\partial \bm n} + \frac{\partial^3  v}{\partial \bm{n}\partial\bm t^2} \Big]\Big\|_{0,e}.
				\label{Local_Eff_Est3_aux2}
		\end{align}}
		Noting that $h_K\eqsim{h_e}$, we have by the trace inequality and
		the inverse inequalities of polynomials that
		\begin{align}
			\|\partial_{\boldsymbol{n}}\phi\|_{0,e}
			\lesssim h_T^{1/2}|\phi|_{2,T} + h_T^{-1/2}|\phi|_{1,T}
			\lesssim h_e^{-3/2}\|\phi\|_{0,T}
			\label{Local_Eff_Est3_aux3}.
		\end{align}
		The desired result follows immediately from  \eqref{Local_Eff_Est1}, \eqref{Local_Eff_Est2},
		\eqref{Local_Eff_Est3_aux1}, \eqref{Local_Eff_Est3_aux2}
		and \eqref{Local_Eff_Est3_aux3}.
	\end{proof}
	
	The following theorem shows the efficiency of the estimator globally, which is a direct consequence of the last theorem.
	\begin{theorem}[Lower bound]
		\label{Thm4_APoster_Err_Est_Lower}
		Given $f\in H^{k-2}(\Omega)$ with $k\ge 2$, let $u$, $u_h$ and $\eta(u_h)$ be given as in Theorem \ref{upper bound}. Then there holds
		\begin{equation}
			\eta(u_h) \lesssim \|u-u_h\|_{h} + \mathcal{O}(h^{k-1}).
			\label{u_uh_H2_lower_Qkh}
		\end{equation}
	\end{theorem}
	\begin{proof}
		It follows from the definition of $\eta(u_h)$ and Theorem \ref{posteriori} that
		\begin{align}\label{etas}
			\eta^2(u_h)\lesssim& \sum_{K\in \mathcal{T}_h} | u-\Pi_h^{\nabla}u_h |_{2,K}^2
			+ J_1(u-u_h,u-u_h)
			\notag\\
			&+\sum_{K\in\mathcal{T}_h}\eta_{4,K}^2+\sum_{K\in\mathcal{T}_h}\eta_{5,K}^2.
		\end{align}
		From \eqref{right}, we obtain
		\begin{equation}\label{etas1}
			\eta_{5,K}=h_K^2\|f-f_h\|_{0,K}\lesssim h^{k+2}\|f\|_{k-2}=\mathcal{O}(h^{k+2}).
		\end{equation}
		By the inverse inequality \eqref{inverse_inequality} and the norm  equivalence \eqref{NormEquivH1},
		\begin{align}\label{etas2}
			| u-\Pi_h^{\nabla}u_h |_{2,K}^2
			& \lesssim | u- u_h |_{2,K}^2 + | u_h-\Pi_h^{\nabla}u_h |_{2,K}^2 \nonumber\\
			& \lesssim | u- u_h |_{2,K}^2
			+h_K^{-2}| u_h-\Pi_h^{\nabla}u_h |_{1,K}^2\nonumber\\
			& \lesssim| u- u_h |_{2,K}^2 +\eta_{4,K}^2.
		\end{align}
		It follows from the norm equivalence in \eqref{NormEquivH1}, the inverse inequality \eqref{inverse_inequality} and  the triangle inequality  that
		\begin{align*}
			\eta_{4,K}& = h_K^{-1} \| \boldsymbol{\chi}(u_h-\Pi_h^{\nabla}u_h) \|_{\ell^2}
			\eqsim h_K^{-1}| u_h-\Pi_h^{\nabla}u_h |_{1,K}
			\lesssim  | u-u_h |_{2,K}  + | u-\Pi_h^{\nabla}u_h |_{2,K}.
		\end{align*}
		Here, by the triangle inequality, Lemma \ref{lem:bound} and \eqref{BHe1}, we derive
		\begin{align*}
			| u-\Pi_h^{\nabla}u_h |_{2,K}
			\le | \Pi_h^{\nabla}u - \Pi_h^{\nabla}u_h |_{2,K} + | u- \Pi_h^{\nabla}u|_{2,K}
			\lesssim  |u - u_h |_{2,K} + h^{k-1}.
		\end{align*}
From the continuity of $u$, the trace inequality and \eqref{BHe1}, we get
\begin{align}\label{j1}
     J_1(u-u_h,u-u_h) &\lesssim   \sum_{e\in\mathcal{E}_h}\frac{1}{|e|}\Big\|\Big[\frac{\partial ( \Pi_h^\nabla u-u)}{\partial \bm{n}_e}\Big]\Big\|_{0,e}^2+\sum_{e\in\mathcal{E}_h}\frac{1}{|e|}\Big\|\Big[\frac{\partial  \Pi_h^\nabla u_h}{\partial \bm{n}_e}\Big]\Big\|_{0,e}^2\notag\\
     &\lesssim \sum_{K\in \mathcal{T}_h}|\Pi_h^\nabla u-u|_{2,K}^2+h_K^{-2}|\Pi_h^\nabla u-u |_{1,K}^2+J_1(u_h,u_h)\notag\\
     &\lesssim h^{2k-2}+J_1(u_h,u_h).
\end{align}
		Inserting \eqref{etas1},\eqref{etas2} and \eqref{j1} into \eqref{etas}, one has
		\begin{align}\label{resu}
			\eta^2(u_h)\lesssim& \sum_{K\in \mathcal{T}_h} | u-u_h |_{2,K}^2+J_1(u_h,u_h)+\mathcal{O}(h^{2k-2}).
		\end{align}
The proof is completed by the definition of the norm $\|\cdot\|_h$.
	\end{proof}
	
\begin{remark}
In the lower bound estimate, the error estimates for $u - \Pi_h^\nabla u$ are not the a priori estimates for the discrete problem.
\end{remark}

	\subsection{Comparison with the analysis of Morley-type virtual elements}

	Ref.~\cite{Carsten2023Morley} provided the a posteriori error analysis for the biharmonic equation using the lowest-order Morley-type or $H^2$-nonconforming virtual elements. However, the local error estimators there do not account for the jumps of the second- and third-order derivatives across interior edges, specifically $\eta_{2,K}$ and $\eta_{3,K}$. As seen in the proof of Theorem \ref{upper bound}, these additional terms arise from integrating by parts for $I_2$ in \eqref{upperI2} and estimating the average term of $J_3$ in \eqref{upperJ3}. For the Morley-type VEM, including $J_i$ ($i=1,2,3$) is unnecessary. Therefore, it is expected to avoid integration by parts of $I_2$ for the Morley-type VEM, which is achieved by means of a special property of the $H^2$-nonconforming interpolation operator described below.
	
	\begin{lemma}\label{lem:H2ncIh}
		Let $V_k^2(K)$ be the local $H^2$-nonconforming virtual element space \cite{Antonietti-Manzini-Verani-2018,Zhao-Zhang-Chen-2018,Carsten2023Morley}. Denote by $I_K^2$ the interpolation operator from $H^2(K)$ to $V_k^2(K)$. Then for any $v\in H^2(K)$ and $w\in V_k^2(K)$, there holds
		\[(\nabla^2 (v - I_K^2v),\nabla^2  w )_K = 0,\]
		This implies that $I_K^2$ is also an $H^2$-elliptic projector from $H^2(K)$ to $V_k^2(K)$.
	\end{lemma}
	\begin{proof}
		See the proof of \cite[Corollary 4.1]{Huang2021YuMedius}.
	\end{proof}

	Our analysis of the upper bound in Theorem \ref{upper bound} can be extended to the Morley-type VEM. Using Lemma \ref{lem:H2ncIh}, we can avoid tedious integration by parts for the term $I_2$ as follows:
	\[
	I_2 = (f_h, \phi - I_h^2 \phi) - \sum_{K \in \mathcal{T}_h} a^K (\Pi_h^{\Delta} u_h, \phi - I_h^2 \phi) = (f_h, \phi - I_h^2 \phi) = (f_h - \Delta^2 \Pi_h^{\Delta} u_h, \phi - I_h^2 \phi),
	\]
	where the last equality holds because $ \Delta^2 \Pi_h^{\Delta} u_h |_K \in \mathbb{P}_{k-4}(K) $ and $ (I_h^2 \phi, p)_K = (\phi, p)_K $ for $ p \in \mathbb{P}_{k-4}(K) $. Consequently, there is no need to include the higher-order jump terms $\eta_{2,K}$ and $\eta_{3,K}$ in the global error estimator for the Morley-type VEMs. However, this does not apply to other types of VEMs for the biharmonic equation even for the $C^1$-continuous VEMs discussed in \cite{ChenHuangLin2022avem}.

	For the Morley-type VEM in \cite{Carsten2023Morley}, unlike the triangle inequality
	\begin{equation}\label{split0}
		|u-u_h|_{2,K} \le |u-E_h u_h|_{2,K} + |u_h - E_h u_h|_{2,K}
	\end{equation}
	used in \eqref{uppereh}, they utilize the following inequality
	\begin{equation}\label{split1}
		|u-u_h|_{2,K} \le |u-\Pi_h^\nabla u_h|_{2,K} + |\Pi_h^\nabla u_h - u_h|_{2,K}.
	\end{equation}
	The second term corresponds to the error estimator $\eta_{4,K}$. For the first term, since $\Pi_h^\nabla u_h|_K$ is a polynomial on $K$, they treat $\Pi_h^\nabla u_h$ as a finite element function on the virtual triangulation $\widehat{\mathcal{T}}$, defining $\widehat{E}_h: \mathbb{P}_2(\widehat{\mathcal{T}}) \to H_0^2(\Omega)$ as the enriching operator for the FEM. They then decompose the first term as
	\begin{equation}\label{split2}
		|u-\Pi_h^\nabla u_h|_{2,K} \le |u- \widehat{E}_h \Pi_h^\nabla u_h|_{2,K} + |\Pi_h^\nabla u_h- \widehat{E}_h \Pi_h^\nabla u_h|_{2,K}.
	\end{equation}
	Let $\widetilde{E}_h = \widehat{E}_h \Pi_h^\nabla$. The decomposition given by \eqref{split1} and \eqref{split2} can be written as
	\begin{align*}
		|u-u_h|_{2,K}
		& \le |u-\widetilde{E}_h u_h|_{2,K} + |u_h - \widetilde{E}_h u_h|_{2,K} \\
		& \le |u-\widetilde{E}_h u_h|_{2,K} + |u_h - \Pi_h^\nabla u_h|_{2,K} + |\Pi_h^\nabla u_h - \widetilde{E}_h u_h|_{2,K}.
	\end{align*}
	The first term can be estimated as for the first term in \eqref{split0}, namely the first term in \eqref{uppereh} for Theorem \ref{upper bound}. By using the standard result for $\widehat{E}_h$, the third term is bounded by
	\[|\Pi_h^\nabla u_h - \widetilde{E}_h u_h|_{2,h}^2 = |\Pi_h^\nabla u_h- \widehat{E}_h \Pi_h^\nabla u_h|_{2,h}^2 \lesssim \sum_{e \in \mathcal{E}_h}
	\Big( \frac{1}{|e|^3} \|[\Pi_h^\nabla v]\|_{0,e}^2 + \frac{1}{|e|} \Big\| \frac{\partial \Pi_h^\nabla v}{\partial \bm{n}_e} \Big\|_{0,e}^2 \Big).\]
	
	For any $ v \in V_h $, it is evident that $ E_h \phi $ can be defined for $ \phi = \Pi_h^\nabla v $. By the definition of $ E_h $, we have $ E_h v = E_h \Pi_h^\nabla v $, which along with Lemma \ref{lem:enriching} implies
	\[
	|\Pi_h^\nabla v - E_h v|_{2,h}^2 \lesssim \sum_{e \in \mathcal{E}_h}
	\Big( \frac{1}{|e|^3} \|[\Pi_h^\nabla v]\|_{0,e}^2 + \frac{1}{|e|} \Big\| \frac{\partial \Pi_h^\nabla v}{\partial \bm{n}_e} \Big\|_{0,e}^2 \Big).
	\]
	Therefore, we can perform the similar decomposition
	\begin{align*}
		|u-u_h|_{2,K}
		& \le |u-E_h u_h|_{2,K} + |u_h - E_h u_h|_{2,K} \\
		& \le |u-E_h u_h|_{2,K} + |u_h - \Pi_h^\nabla u_h|_{2,K} + |\Pi_h^\nabla u_h - E_h u_h|_{2,K}.
	\end{align*}
	This approach allows us to directly derive the error estimators in \cite{Carsten2023Morley} for the Morley VEMs of arbitrary order.

\section{Numerical experiments} \label{Sec:numerical}
	
In this section, we present several benchmark tests for which the exact solution is known. These tests illustrate the robustness of the residual-type a posteriori error estimator and demonstrate the efficiency of the adaptive VEM. Unless otherwise specified, we only consider the lowest-order element $k=2$. We also provide the treatment of inhomogeneous boundary value conditions.

\subsection{Treatment of inhomogeneous boundary value conditions}

Consider the fourth-order problem with inhomogeneous boundary value conditions
	\begin{equation}\label{strongform1}
		\left\{\begin{array}{ll}
			\Delta^2u = f & {\rm in}~~\Omega,\\
			u = g_D, \quad \dfrac{\partial u}{\partial \bm n}=g_N& {\rm on}~~\partial \Omega,
		\end{array}\right.
	\end{equation}
where $g_D \in H^{3/2}(\partial\Omega)$ and $g_N \in H^{1/2}(\partial\Omega)$ are prescribed functions. Let $v\in H^2(K)$. The integration by parts gives
	\begin{align*}
		& (f,v)_K = ( \Delta^2 u, v)_K
		= (\nabla^2 u, \nabla^2 v)_K  \\
		& \hspace{3cm} + \int_{\partial K} \Big( \frac{\partial \Delta u}{\partial \bm{n}_K} v  + \frac{\partial^2 u}{\partial \bm{n }_K \partial \bm{t}_K} \frac{\partial v}{\partial \bm{t}_K} \Big) \d s
		- \int_{\partial K} \frac{\partial^2 u}{\partial \bm{n}_K^2}\frac{\partial v}{\partial \bm{n}_K} \d s.
	\end{align*}
	If $ v $ is $ C^0 $-continuous, the second term on the right-hand side naturally vanishes when we sum the equation over all polygons in $ \mathcal{T}_h $. However, when $ v = v_h $ is a function in the interior penalty virtual element space $ V_h $, which is not continuous, the second term contributes to the summation. The IPVEM in \cite{ZMZW2023IPVEM} neglects this contribution as it introduces a consistency error with a desired convergence order as shown in \cite{ZMZW2023IPVEM,FY2023IPVEM}, which gives
	\begin{align*}
		(f,v_h) \approx (\nabla^2 u, \nabla^2 v_h) - \sum_{K\in \mathcal{T}_h} \int_{\partial K} \frac{\partial^2 u}{\partial \bm{n}_K^2}\frac{\partial v_h}{\partial \bm{n}_K} \d s.
	\end{align*}
	Using the formula in \eqref{magic} and noting that $u \in H^3(\Omega)$, we get
	\begin{align*}
		\sum_{K\in \mathcal{T}_h} \int_{\partial K} \frac{\partial^2 u}{\partial \bm{n}_K^2}\frac{\partial v_h}{\partial \bm{n}_K} \d s
		& = \sum_{e\in \mathcal{E}_h} \int_e \Big\{\frac{\partial^2 u}{\partial \bm{n}_e^2}\Big\}\Big[\frac{\partial v_h}{\partial \bm{n}_e}\Big] \d s
		+ \sum_{e\in \mathcal{E}_h^0} \int_e \Big[\frac{\partial^2 u}{\partial \bm{n}_e^2}\Big]\Big\{\frac{\partial v_h}{\partial \bm{n}_e}\Big\} \d s \\
		& = \sum_{e\in \mathcal{E}_h} \int_e \Big\{\frac{\partial^2 u}{\partial \bm{n}_e^2}\Big\}\Big[\frac{\partial v_h}{\partial \bm{n}_e}\Big] \d s,
	\end{align*}
	yielding
	\[(f,v_h) \approx (\nabla^2 u, \nabla^2 v_h) - \sum_{e\in \mathcal{E}_h} \int_e \Big\{\frac{\partial^2 u}{\partial \bm{n}_e^2}\Big\}\Big[\frac{\partial v_h}{\partial \bm{n}_e}\Big] \d s.\]
	To ensure the symmetry and enforce the $C^1$-continuity of the discrete variational problem, we rewrite it as
	\begin{align*}
		& (f,v_h)+\sum_{e\in \mathcal{E}_h^\partial} \Big(-\int_e \Big\{\frac{\partial^2 v_h}{\partial \bm{n}_e^2}\Big\}\Big[\frac{\partial u}{\partial \bm{n}_e}\Big] \d s  + \frac{\lambda_e^3}{|e|^3}\int_e [u][v_h] \d s + \frac{\lambda_e}{|e|}\int_e \Big[\frac{\partial u}{\partial \bm{n}_e}\Big]\Big[\frac{\partial v_h}{\partial \bm{n}_e}\Big] \d s \Big)\\
		& \approx (\nabla^2 u, \nabla^2 v_h) - \sum_{e\in \mathcal{E}_h} \int_e \Big\{\frac{\partial^2 u}{\partial \bm{n}_e^2}\Big\}\Big[\frac{\partial v_h}{\partial \bm{n}_e}\Big] \d s
		- \sum_{e\in \mathcal{E}_h} \int_e \Big\{\frac{\partial^2 v_h}{\partial \bm{n}_e^2}\Big\}\Big[\frac{\partial u}{\partial \bm{n}_e}\Big] \d s \\
		& \qquad + \sum_{e\in \mathcal{E}_h} \frac{\lambda_e^3}{|e|^3}\int_e [u][v_h] \d s  + \sum_{e\in \mathcal{E}_h} \frac{\lambda_e}{|e|}\int_e \Big[\frac{\partial u}{\partial \bm{n}_e}\Big]\Big[\frac{\partial v_h}{\partial \bm{n}_e}\Big] \d s ,
	\end{align*}
	where we have used the fact that $[u]_e = 0$ and $[\frac{\partial u}{\partial \bm{n}_e}]_e = 0$ for $e \in \mathcal{E}_h^0$. $\lambda_e\ge 1$ is some edge-dependent parameter. Taking into account Remark \ref{rem:C0issue}, we can describe the IPVEM for \eqref{strongform1} with inhomogeneous boundary value conditions as follows: Find $u_h \in V_h^g = \{ v\in V_h: v = I_h g_D ~{\rm on}~\partial \Omega\}$ such that
	\begin{equation}\label{IPVEM1}
		\mathcal{A}_h(u_h,v_h) = F_h(v_h),\quad v_h\in V_h,
	\end{equation}
	where $g_D$ has been extended to the whole domain.  The right-hand side is
	\[F_h(v_h) = \langle f_h,v_h\rangle + \sum_{e\in \mathcal{E}_h^\partial} \int_e\Big(
	-g_N\frac{\partial^2 \Pi_h^\nabla v_h}{\partial \bm{n}_e^2}
	+ g_N\frac{\lambda_e}{|e|} \frac{\partial \Pi_h^\nabla v_h}{\partial \bm{n}_e}
	\Big)\d s.\]

	Standard adaptive algorithms based on the local mesh refinement can be written as loops of the form
	\[{\bf SOLVE} \to {\bf ESTIMATE} \to {\bf MARK} \to {\bf REFINE}.\]
	Given an initial polygonal subdivision $\mathcal{T}_0$, to get $\mathcal{T}_{k+1}$ from $\mathcal{T}_k$ we
	first solve the VEM problem under consideration to get the numerical solution $u_k$ on $\mathcal{T}_k$. The error is then estimated by using $u_k$, $\mathcal{T}_k$ and the a posteriori error bound. The local error bound is used to mark a subset
	of elements in $\mathcal{T}_k$ for refinement. The marked polygons and possible more neighboring elements are refined
	in such a way that the subdivision meets certain conditions, for example, the resulting polygonal mesh is still shape regular.

	In the implementation, all the initial meshes are generated by the PolyMesher package \cite{Talischi-Paulino-Pereira-2012},
	which is a polygonal mesh generator based on the Centroidal Voronoi Tessellations.
	In the {\bf MARK} step, We employ the D\"{o}rfler marking strategy with parameter $\theta\in (0,1)$ to select the subset of elements for refinement.  In the {\bf REFINE} step, it is usually time-consuming to write a mesh refinement function since we need to
	carefully design the rule for dividing the marked elements to get a refined mesh of high quality. We divide every polygonal element by connecting the midpoint of each edge to its barycenter, which requires that the barycenter is an internal point of each element. Moreover, for any two polygons $K_1$ and $K_2$ sharing a common edge $e$, we refine them together if $K_1$ is in the refinement set and one endpoint of $e$ is the hanging node of $K_2$, in order to avoid occurrence of short-edges. Such an implementation can be found in \cite{Yu-2021}.

	All examples are implemented in MATLAB R2019b. Our code is available from GitHub (\url{https://github.com/Terenceyuyue/mVEM}) as part of the mVEM package which contains efficient and easy-following codes for various VEMs published in the literature.
	The subroutine
	IPVEM.m is used to compute the numerical solutions.
	The test script main\_IPVEM\_avem.m verifies the convergence rates. The error estimator is computed by the subroutine IPVEM\_indicator.m.
	
Since the VEM solution $u_h$ is not explicitly known inside the polygonal elements, we evaluate the projection error, which is quantified by
	\[
	{\rm ErrH2} = | u - \Pi_h^{\Delta}u_h  |_{2,h} =
	\big( \sum\limits_{K\in\mathcal{T}_h} | u - \Pi_h^{\Delta}u_h |_{2,K}^2 \big)^{1/2}.
	\]

	\subsection{Solution without large gradients}
	
	\begin{example}
		Let $\Omega=(0,1)^2$ and we select the load term $f$ such that the analytical solution of the problem \eqref{BilinearAh} is
		\begin{itemize}
			\item homogeneous boundary value conditions:
			\[u(x,y) = 10 x^2 y^2 (1-x)^2 (1-y)^2 \sin(\pi x).\]
			\item or inhomogeneous boundary value conditions:
			\[u(x,y) = 10 x^2 y^2 (1-x)^2 (1-y)^2 \sin(\pi x) + x^2 + y^2.\]
		\end{itemize}
		
	\end{example}
	
	\begin{figure}[H]
		\centering
		\includegraphics[scale=0.15]{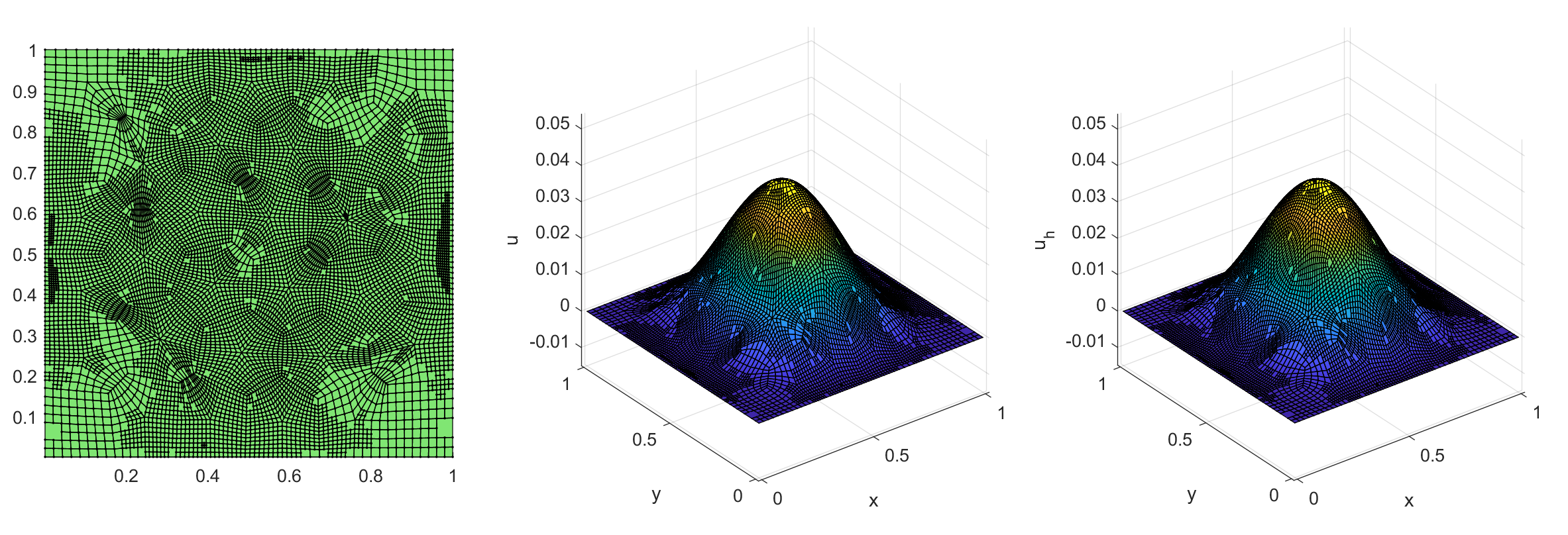}\\
		\caption{The final adapted mesh and the exact and numerical solutions with homogeneous boundary values conditions}\label{fig:ex1solu}
	\end{figure}
	
	\begin{figure}[H]
		\centering
		\subfigure[]{\includegraphics[scale=0.35]{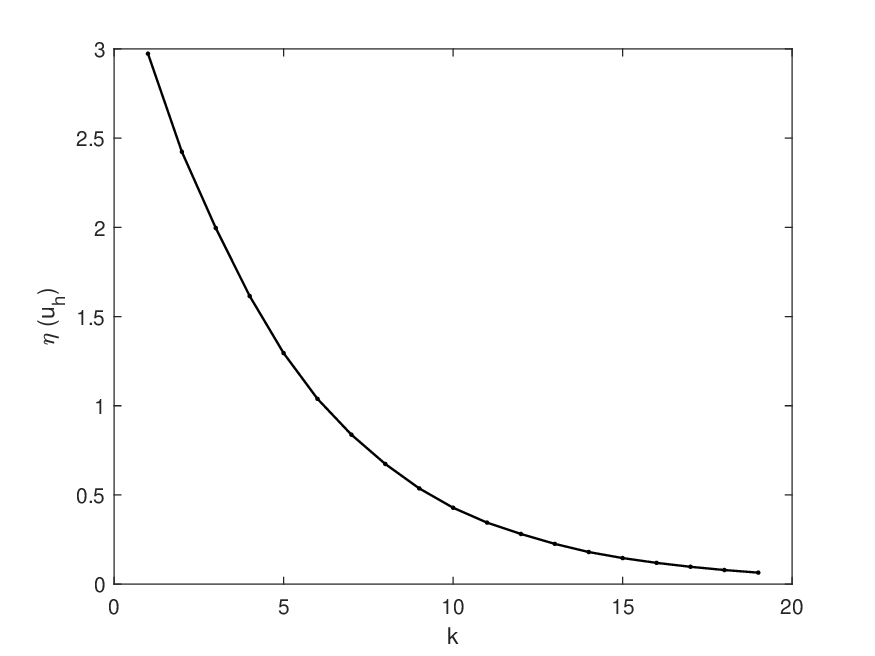}}
		\subfigure[]{\includegraphics[scale=0.35]{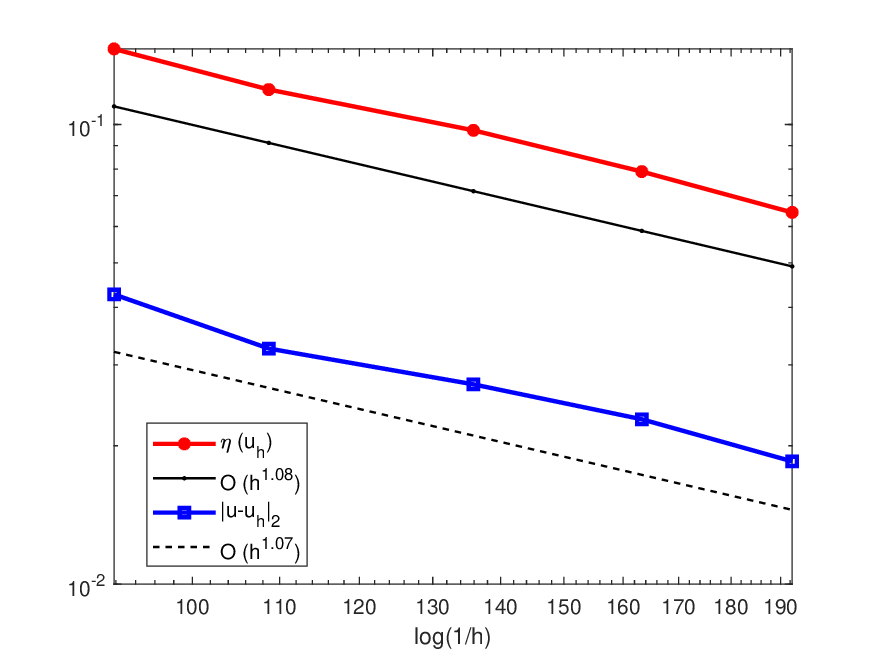}}\\
		\caption{(a) The estimator $\eta(u_h, \mathcal{T}_h)$ for adapted meshes; (b) The convergence orders }\label{fig:ex1estimator}
	\end{figure}

In this example, we employ the D\"orfler marking strategy with a parameter of $\theta = 0.4$ to select the subset of elements for refinement. The final adapted mesh and the solutions are displayed in Fig.~\ref{fig:ex1solu} for the problem with homogeneous boundary values conditions.
Since the exact solution is sufficiently smooth without large gradients, the initial mesh is refined almost everywhere. The estimator $\eta(u_h, \mathcal{T}_h)$ for all adapted meshes is plotted in Fig.~\ref{fig:ex1estimator}a, showing a decrease throughout the adaptive procedure as expected. The convergence order is also displayed in Fig.~\ref{fig:ex1estimator}b, where we observe the optimal first-order convergence rate for both the a priori error and the error estimator, consistent with theoretical predictions.
	Almost the same results are observed for this problem with inhomogeneous boundary values conditions. The final adapted mesh and the solutions are displayed in Fig.~\ref{fig:ex1soluinhom}.
	
	\begin{figure}[H]
		\centering
		\includegraphics[scale=0.15]{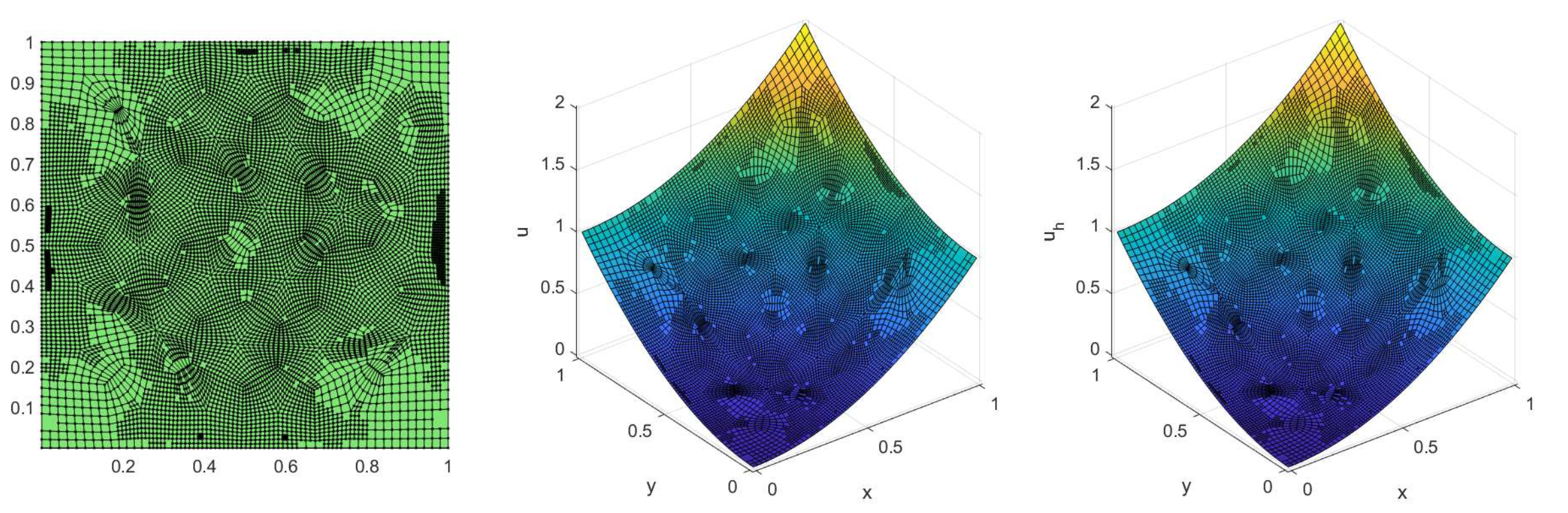}\\
		\caption{The final adapted mesh and the exact and numerical solutions with inhomogeneous boundary values conditions}\label{fig:ex1soluinhom}
	\end{figure}

	\subsection{Solution with large gradients}

	\begin{example}\label{exam:Ex2}
		Let $\Omega=(0,1)^2$ and we select the load term $f$ such that the analytical solution of the problem \eqref{BilinearAh} is
		\[u(x,y) = xy(1 - x)(1 - y){\mathrm{exp}}\left(  - 1000((x - 0.5)^2 + (y - 0.117)^2) \right).\]
	\end{example}

	\begin{figure}[H]
		\centering
		\subfigure[]{\includegraphics[scale=0.45]{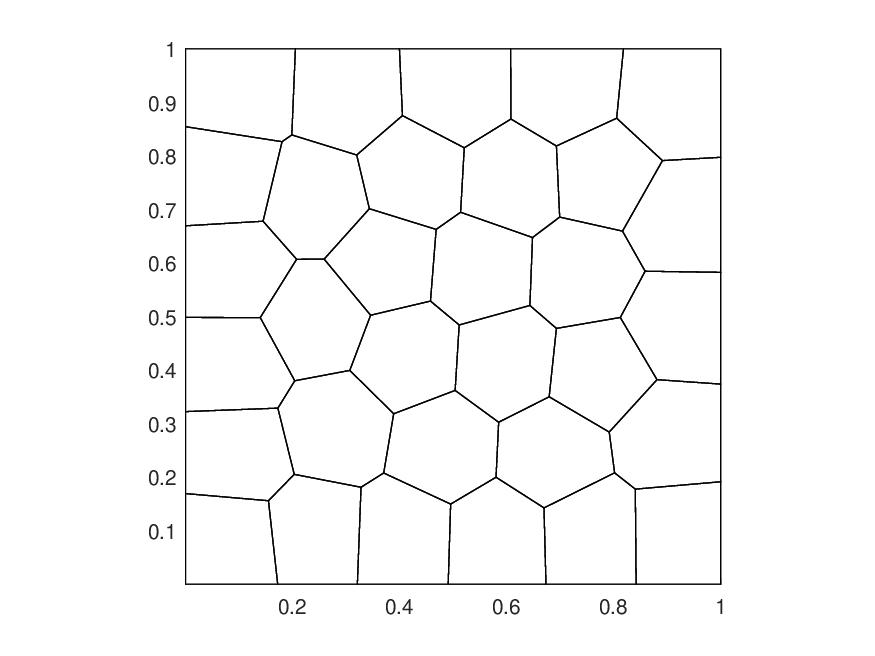}}
		\subfigure[]{\includegraphics[scale=0.45]{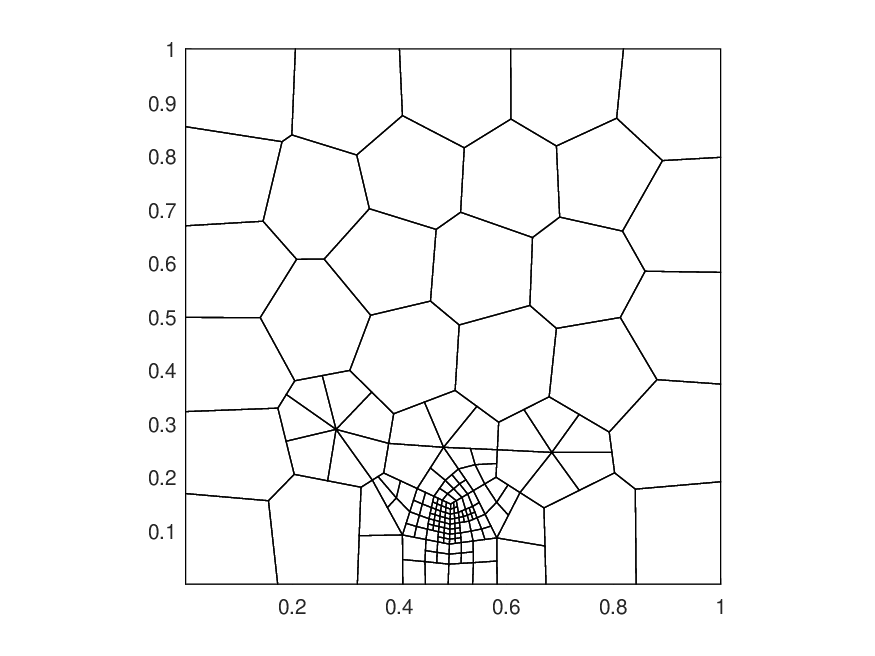}}\\
		\subfigure[]{\includegraphics[scale=0.45]{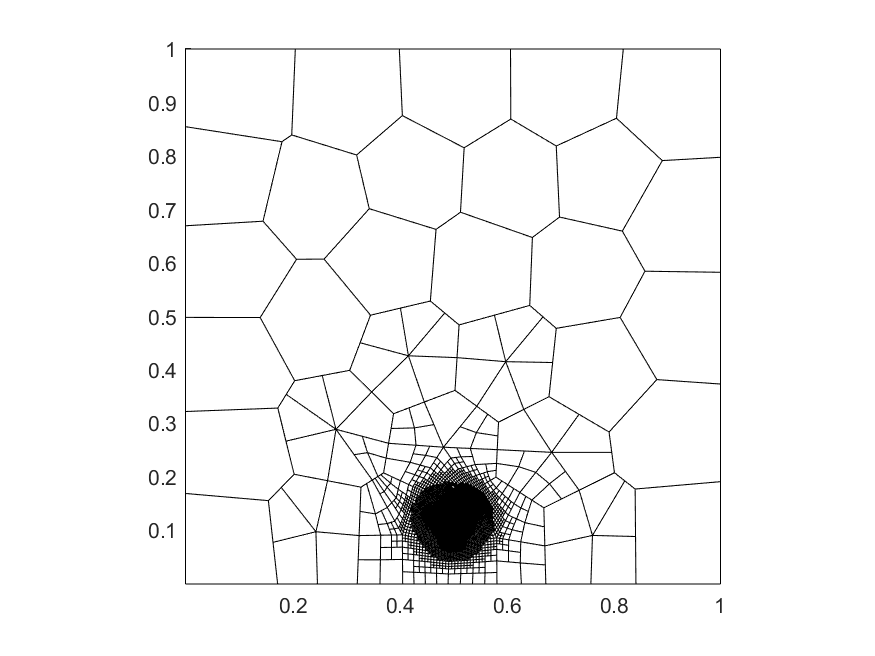}}
		\subfigure[]{\includegraphics[scale=0.45]{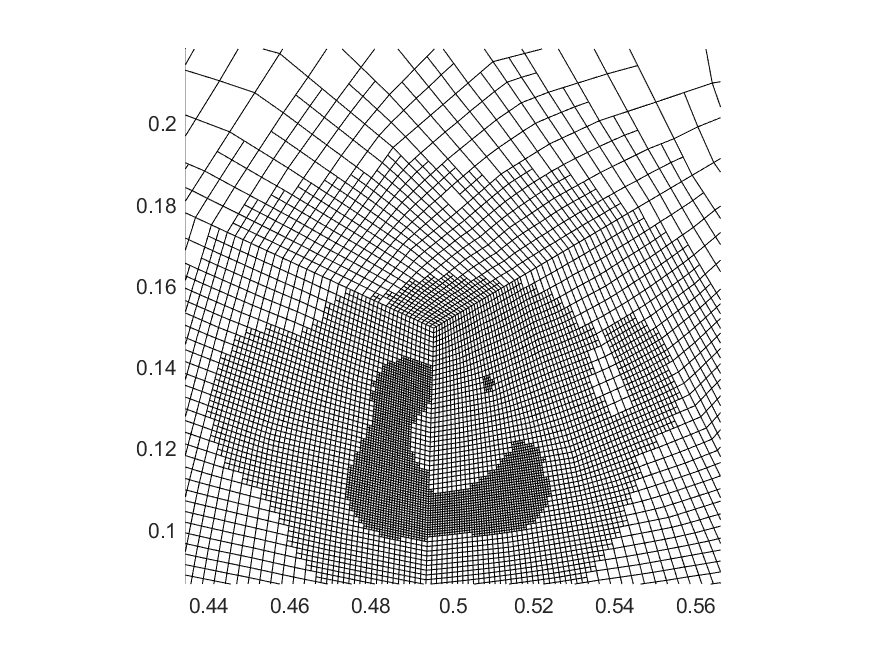}}
		\caption{The initial and the final adapted meshes. (a) The initial mesh;
			(b) After 10 refinement steps; (c) After 20 refinement steps; (d) The zoomed mesh in (c)}\label{refineVEMmesh}
	\end{figure}
	
	\begin{figure}[H]
		\centering
		\includegraphics[scale=0.5]{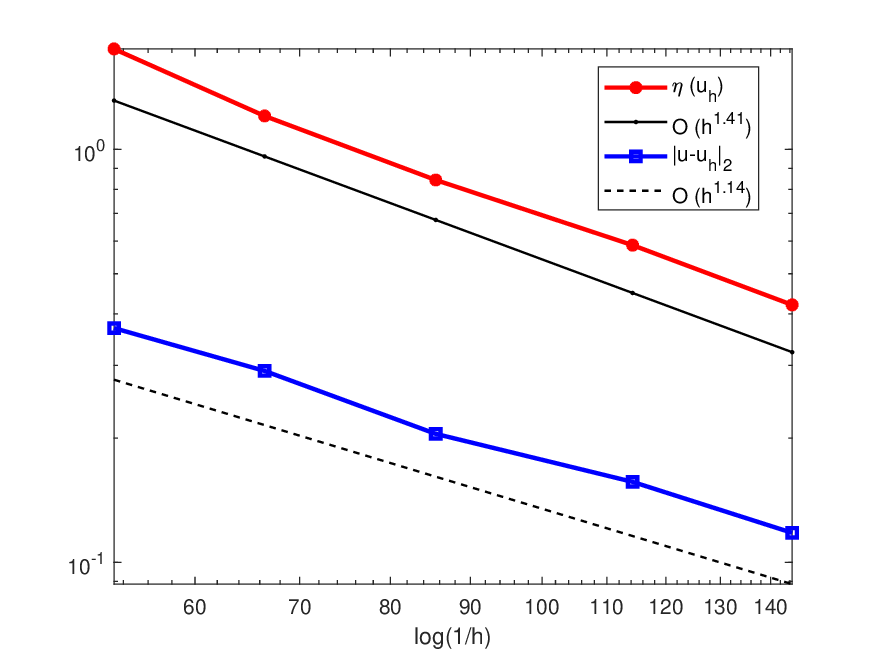}\\
		\caption{Convergence rates of the error $|u - \Pi_h^\Delta u_h|_{2,h}$ and the error estimator $\eta(u_h)$}\label{fig:adaptiveRat}
	\end{figure}
	
	\begin{figure}[H]
		\centering
		\subfigure[Exact]{\includegraphics[scale=0.15,trim = 450 0 450 0,clip]{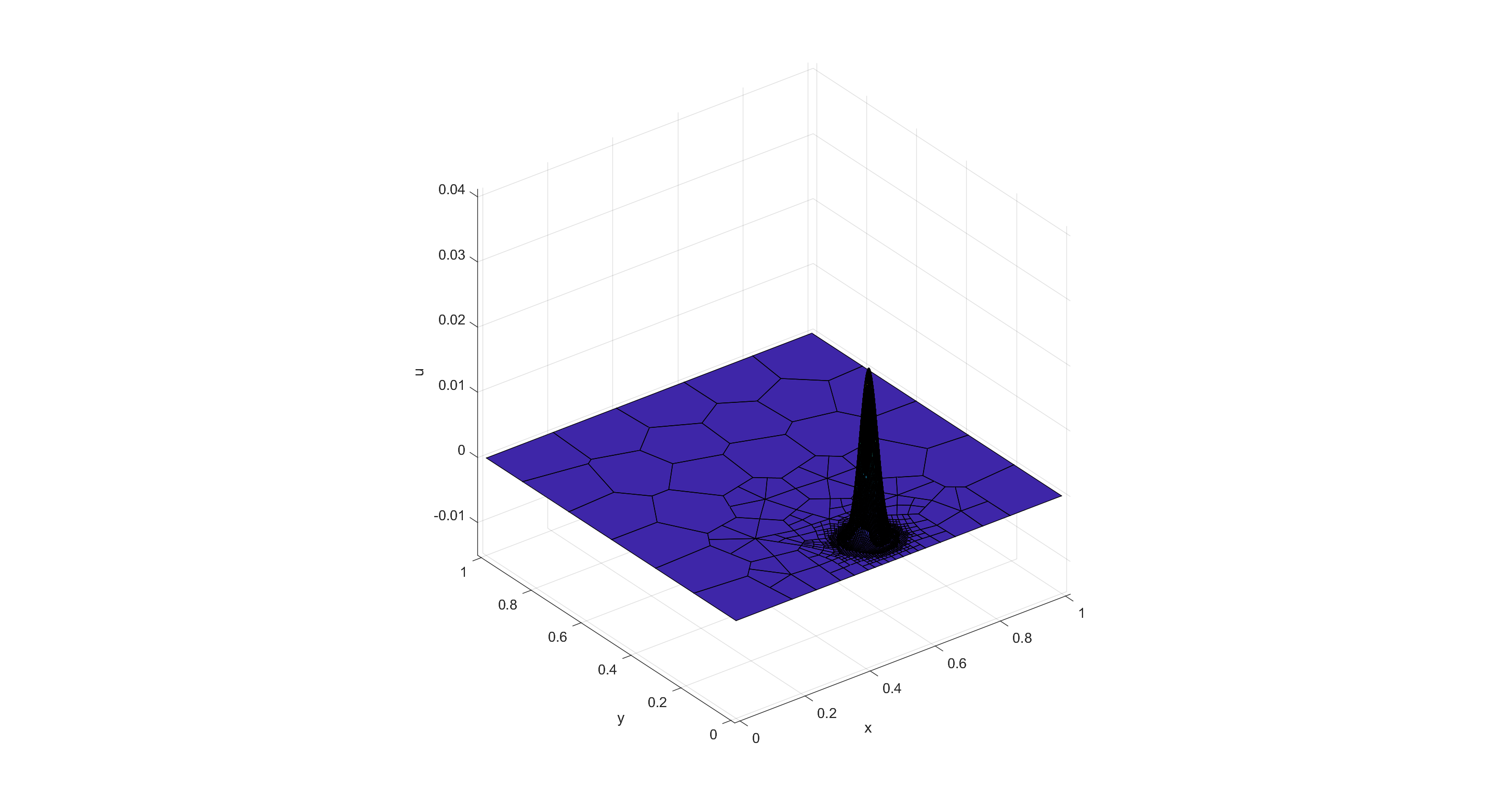}} \hspace{-2em}
		\subfigure[Numerical]{\includegraphics[scale=0.15,trim = 450 0 450 0,clip]{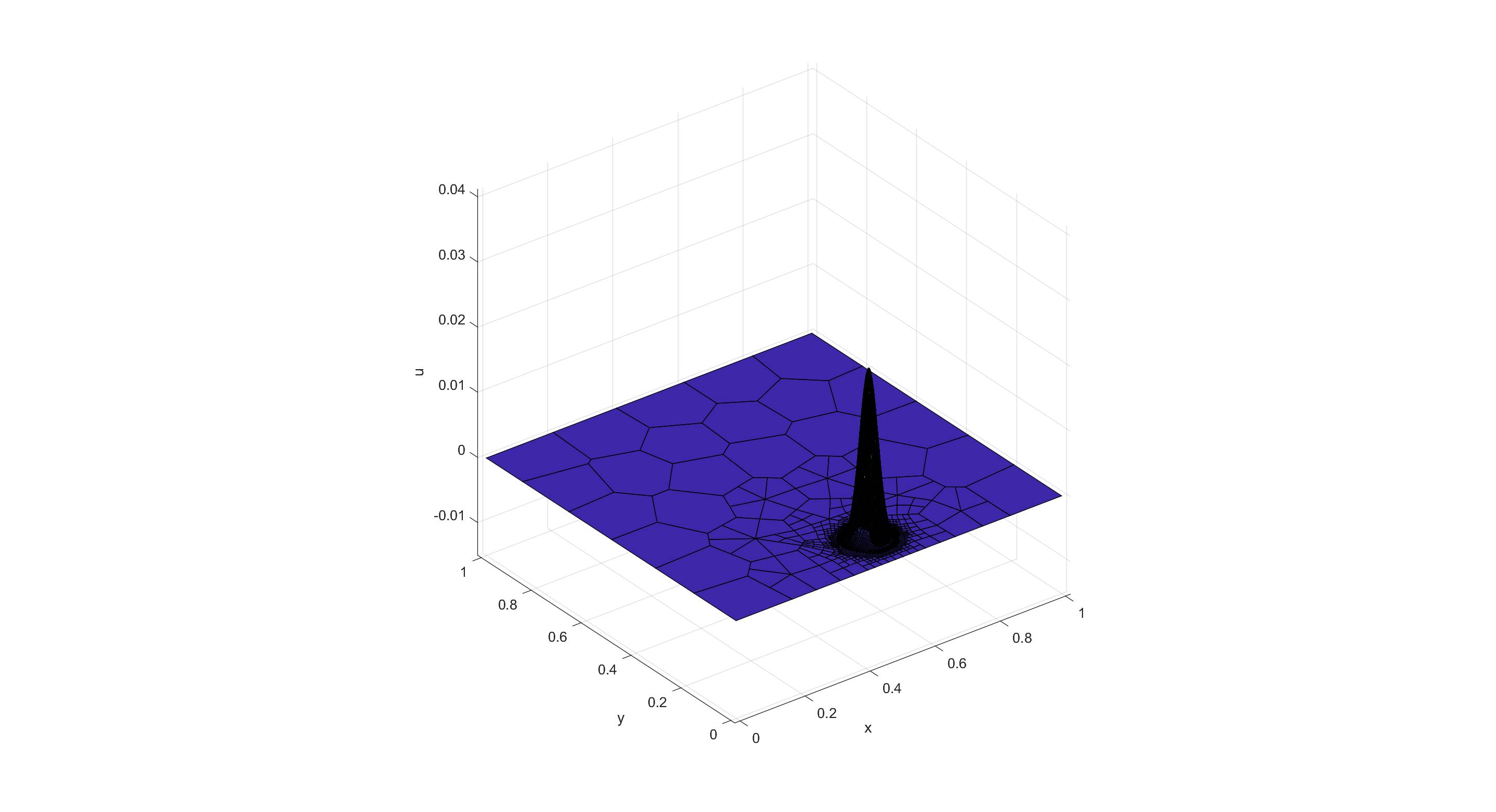}}\\
		\caption{The exact and numerical solutions for Example \ref{exam:Ex2}}\label{fig:refineVEMsol}
	\end{figure}

	For this example, we choose the parameter $\theta = 0.6$.
	The initial mesh and the final adapted meshes after 10 and 20 refinement steps are presented in Fig.~\ref{refineVEMmesh}~(a-c),
	respectively. The detail of the last mesh is shown in Fig.~\ref{refineVEMmesh}~(d). Clearly, no small edges are observed. We plot the adaptive convergence order in Fig.~\ref{fig:adaptiveRat}, which confirms the theoretical prediction as in the last example.
	We also display the numerical and exact solutions in Fig.~\ref{fig:refineVEMsol}, from which we see the adaptive strategy correctly refines the mesh in a neighborhood of the singular point and there is a good level of agreement between the $H^2$ error and error estimator.

\subsection{Solution with low regularity}

	\begin{figure}[H]
		\centering
		\includegraphics[scale=0.5]{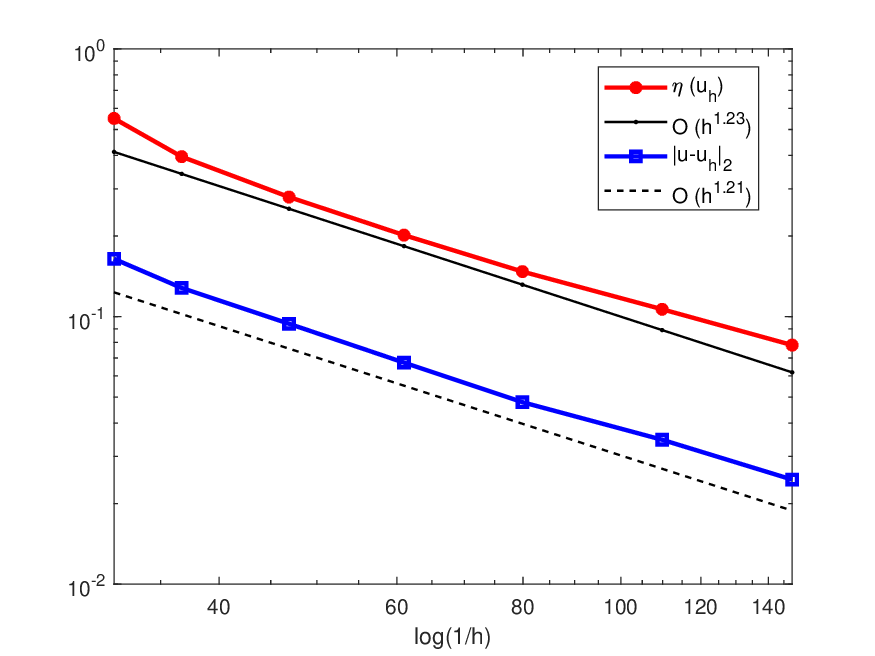}\\
		\caption{Convergence rates of the error $|u - \Pi_h^\Delta u_h|_{2,h}$ and the error estimator $\eta(u_h)$}\label{fig:Ex3order}
	\end{figure}

\begin{example}\label{exam:Ex3}
This example examine the performance of the IPVEM under the low regularity condition. The domain is takes as an L-shape square $\Omega = (0,1)^2\backslash ([1/2,1]\times [0,1/2])$. We select the load term $f$ such that the analytical solution is
	\[u(x,y) = ( (x - 0.5)^2 + (y - 0.5)^2 )^{5/6}.\]
\end{example}

It is evident that the solution $ u \in H^{8/3 - \epsilon} $ for some $ \epsilon > 0 $ exhibits singular behavior at the point $ (1/2, 1/2) $. For error evaluation we introduce a small parameter $ \delta $ into the exact solution. Specifically, we modify $ u $ as $ u(x,y) = ( (x - 0.5)^2 + (y - 0.5)^2 + \delta )^{5/6} $. In this example, we choose the parameter $ \theta = 0.6 $. We plot the adaptive convergence order in Fig.~\ref{fig:Ex3order}, from which we observe that the IPVEM achieves the correct convergence order, even under the low regularity condition of the solution. Additionally, we present the numerical and exact solutions in Fig.~\ref{fig:Ex3solu}, which demonstrate that the adaptive strategy effectively refines the mesh in the vicinity of the singular point.

	\begin{figure}[H]
		\centering
		\includegraphics[scale=0.2]{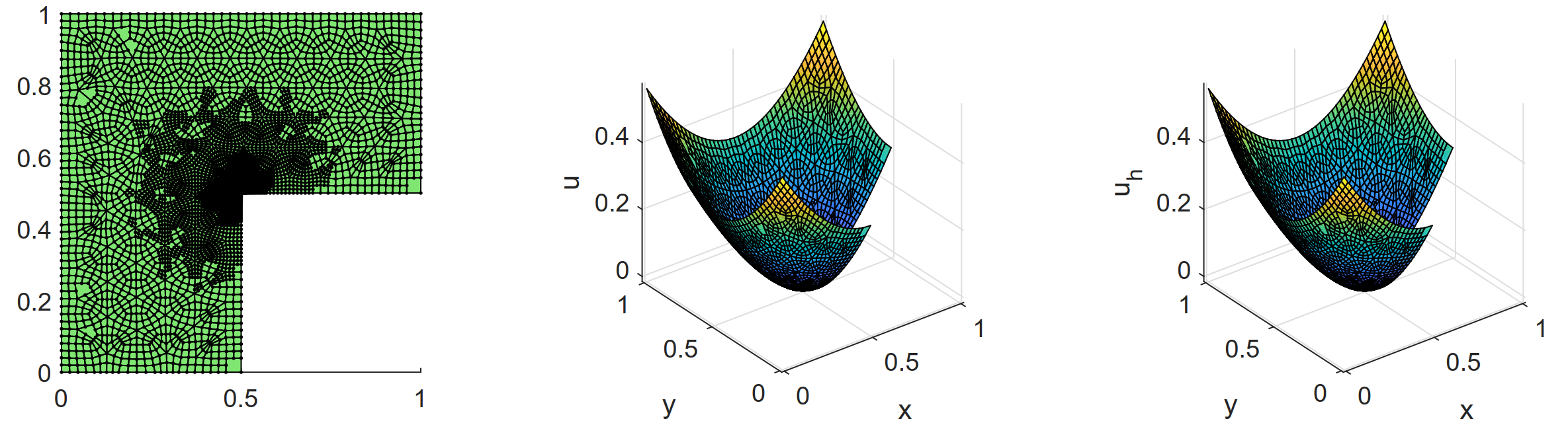}\\
		\caption{The exact and numerical solutions for Example \ref{exam:Ex3}}\label{fig:Ex3solu}
	\end{figure}

\subsection{Formulation based on $H^2$-elliptic projector}

In \cite{Carsten2023Morley,ChenHuangLin2022avem}, the $ H^1 $-elliptic projector $ \Pi_h^\nabla $ is replaced by the $ H^2 $-elliptic projector $ \Pi_h^\Delta $ in the error estimators. For the IPVEM, we can adopt a similar approach, except for the handling of gradient jumps. In this case,
we define
	\begin{align*}
	& J_1(v,w)=   \sum_{e\in\mathcal{E}_h}\frac{\lambda_e}{|e|}\int_e\Big[\frac{\partial  \Pi_h^\Delta v}{\partial \bm{n}_e}\Big]\Big[\frac{\partial \Pi_h^\Delta w}{\partial \bm{n}_e}\Big] \d s, \\
    & J_2(v,w)=-\sum_{e \in \mathcal{E}_h}\int_e\Big\{\frac{\partial^2\Pi_h^\Delta v}{\partial \bm{n}_e^2}\Big\}\Big[\frac{\partial \Pi_h^\nabla w}{\partial \bm{n}_e}\Big] \d s, \\
    & J_3(v,w)=-\sum_{e\in\mathcal{E}_h}\int_e\Big\{\frac{\partial^2 \Pi_h^\Delta w}{\partial \bm{n}_e^2}\Big\}\Big[\frac{\partial  \Pi_h^\nabla v}{\partial \bm{n}_e}\Big] \d s.
	\end{align*}
The substitution of the $ H^1 $ projection in the gradient jumps creates difficulty in the process of $ I_{2,3} $ in \eqref{I23}, as we cannot cancel $ J_2 $. On the other hand, using the $ H^1 $ projection is the preferred choice in the a priori estimate in \cite{FY2023IPVEM}.

We can still prove that $\|w\|_h  = (|w|_{2,h}^2+J_1(w,w))^{1/2}$ introduces a norm on $V_h$. It is enough to prove that $\|v_h\|_h = 0$ implies $v_h = 0$ for any given $v_h \in V_h$. By definition, $\|v_h\|_h = 0$ is equivalent to $|v_h|_{2,h} = 0 $ and $J_1(v_h,v_h) = 0$. Equation $|v_h|_{2,h} = 0$ shows that $\nabla_h v_h$ is a piecewise constant on $\mathcal{T}_h$. On the other hand, the direct manipulation yields
\begin{align*}
\int_e\Big[\frac{\partial v_h}{\partial \bm{n}_e}\Big]^2 \d s
& \le \int_e\Big[\frac{\partial  \Pi_h^\Delta v_h}{\partial \bm{n}_e}\Big]^2  \d s
+ \int_e\Big[\frac{\partial  (v_h-\Pi_h^\Delta v_h)}{\partial \bm{n}_e}\Big]^2  \d s = \int_e\Big[\frac{\partial  (v_h-\Pi_h^\Delta v_h)}{\partial \bm{n}_e}\Big]^2  \d s,  \\
& \le C \sum_{K= K^-, K^+} (|v_h-\Pi_h^\Delta v_h|_{1,K} + |v_h-\Pi_h^\Delta v_h|_{2,K}) \le C\sum_{K= K^-, K^+} |v_h|_{2,K} = 0.
\end{align*}
Since $\nabla_h v_h$ is piecewise constant, we further obtain $[\frac{\partial v_h}{\partial  \bm{n}_e}] = 0$ over the edges of $\mathcal{T}_h$. That is, the normal derivative of $v_h$ is continuous at the interior edges and vanishes at the boundary of $\Omega$. This reduces to the discussion in the proof of Lemma 4.2 in \cite{ZMZW2023IPVEM}, so we omit the remaining argument.

We repeated the test for all the examples and obtained similar results, so we do not include them here.
The test script main\_IPVEM\_avemH2.m verifies the convergence rates. The error estimator is computed by the subroutine IPVEM\_indicatorH2.m.

%

\section*{CRediT authorship contribution statement}
		
All authors worked closely together on the conceptualization, methodology, and writing of this research. Additionally, Yue Yu and Yuming Hu were responsible for implementing the discrete method used in the study.
		
\section*{Declaration of competing interest}
		
The authors declare that they have no known competing financial interests or personal relationships that could have appeared to
influence the work reported in this paper.
		
\section*{Data availability}
		
		No data was used for the research described in the article.
		
		\section*{Acknowledgements}
		
Fang Feng was partially supported by the National Natural Science Foundation of China (NSFC) grant  12401528 and
the Fundamental Research Funds for the Central Universities, No. 30924010837.
Yue Yu was partially supported by the National Science Foundation for Young Scientists of China (No. 12301561) and
the Key Project of Scientific Research Project of Hunan Provincial Department of Education (No. 24A0100).
Jikun Zhao was partially supported by National Natural Science Foundation of China (No. 12371411).

		
		\addcontentsline{toc}{section}{References}
		\bibliographystyle{plain} 
		\bibliography{Refs_IPVEM}

	\end{document}